%% file: Topological_properties_of_split-G2_SL3C_and_SL3R2_forms_on_manifolds.tex
\documentclass[10pt,reqno]{amsart}
\pdfoutput=1
\input{Festino_1pt0}

\title[Topological properties of closed \sg, \(\SL(3;\bb{C})\) and \(\SL(3;\bb{R})^2\) forms on manifolds]{\boldmath Topological properties of closed \sg, \(\SL(3;\bb{C})\) and \(\SL(3;\bb{R})^2\) forms on manifolds}
\author{Laurence H. Mayther}

\begin{document}\fontsize{10pt}{12pt}\selectfont
\begin{abstract}
\footnotesize{This paper uses algebro-topological techniques such as characteristic classes and obstruction theory, together with the \(h\)-principles for \sg\ and \slr\ forms recently established by the author and the \(h\)-principle for \slc\ forms established by Donaldson, to prove results on the topological properties of closed \sg, \slc\ and \slr\ forms on oriented 6- and 7-manifolds.  Specifically, a criterion for an arbitrary oriented 7-manifold to admit a closed (resp.\ coclosed) \sg-structure is obtained, proving a conjecture of L\^{e}; a generalisation of Donaldson's `\g-cobordisms' to \sg, \slc\ and \slr\ forms is introduced, with homotopic \slc\ and \slr\ forms in a given cohomology class shown to be \sg-cobordant, a result which currently has no analogue in the \g\ case; and a complete classification of closed \slc\ forms up to homotopy is provided.  Additionally, a lower bound on the number of homotopy classes of closed \slr\ forms on a given manifold is obtained, and the question of which closed \slc\ or \slr\ forms arise as the boundary values of closed \sg-structures on oriented 7-manifolds is investigated.}
\end{abstract}
\maketitle

\section{Introduction}

Let \(\si_0 \in \ww{p}\lt(\bb{R}^n\rt)^*\) be stable in the sense of \cite{SF&SM}, i.e.\ the \(\GL_+(n;\bb{R})\) orbit of \(\si_0\) in \(\ww{p}\lt(\bb{R}^n\rt)^*\) is open, and let \(\M\) be an oriented \(n\)-manifold.  A \(p\)-form \(\si \in \Om^p(\M)\) is termed a \(\si_0\)-form if, for each \(x \in \M\), there exists an orientation-preserving isomorphism \(\al: \T_x\M \to \bb{R}^n\) satisfying \(\al^*\si_0 = \si\); by stability, any sufficiently small perturbation of \(\si\) is also a \(\si_0\)-form.  A closed \(\si_0\)-form is then simply a \(\si_0\)-form \(\si\) which satisfies \(\dd\si = 0\).  The aim of this paper is to study the topological properties of the space of closed \(\si_0\)-forms on an arbitrary oriented \(n\)-manifold for the following four choices of stable form \(\si_0\):
\cag
\svph_0 = \th^{123} - \th^{145} - \th^{167} + \th^{246} - \th^{257} - \th^{347} - \th^{356} \in\ww{3}\lt(\bb{R}^7\rt)^*;\\
\svps_0 = \th^{4567} - \th^{2367} - \th^{2345} + \th^{1357} - \th^{1346} - \th^{1256} - \th^{1247} \in\ww{4}\lt(\bb{R}^7\rt)^*;\\
\rh_- = \th^{135} - \th^{146} - \th^{236} - \th^{245} \in \ww{3}\lt(\bb{R}^7\rt)^*;\\
\rh_+ = \th^{123} + \th^{456} \in \ww{3}\lt(\bb{R}^7\rt)^*.
\caag
The corresponding \(\si_0\)-forms are known as \sg\ 3-forms, \sg\ 4-forms, \slc\ 3-forms and \slr\ 3-forms respectively, the names referring to the respective stabilisers of the forms \(\svph_0\), \(\svps_0\), \(\rh_-\) and \(\rh_+\) in \(\GL_+(7;\bb{R})\) and \(\GL_+(6;\bb{R})\) as appropriate, where \sg\ denotes the centreless, doubly connected Lie group corresponding to the split real form of the exceptional Lie algebra \(\fr{g}_{2,\bb{C}}\).

Let \(\si_0 = \svph_0\), \(\svps_0\), \(\rh_+\) or \(\rh_-\).  Given an oriented \(n\)-manifold \(\M\) and a fixed cohomology class \(\al \in \dR{p}(\M)\) (where \(n = 6\), \(7\) and \(p = 3\), \(4\), as appropriate), write \(\CL^p_{\si_0}(\M)\) for the set of closed \(\si_0\)-forms on \(\M\) and \(\CL^p_{\si_0}(\al)\) for the set of closed \(\si_0\)-forms representing the cohomology class \(\al\).  In  \cite{RoG2MwB}, Donaldson proved that the inclusions:
\e\label{HE}
\CL^p_{\si_0}(\al) \emb \CL^p_{\si_0}(\M) \emb \Om^p_{\si_0}(\M)
\ee
are homotopy equivalences in the case where \(\si_0 = \rh_-\).  Using different techniques, in \cite{RhPfCSF, TRhPfCSL3R23F} the author proved that corresponding inclusions  in the three cases \(\si = \svph_0\), \(\svps_0\) and \(\rh_+\) are homotopy equivalences, as well as providing an alternative proof of Donaldson's result.  By virtue of these homotopy equivalences, the topological properties of the spaces of \sg\ 3- and 4-forms, \slc\ 3-forms and \slr\ 3-forms which are closed, or which lie in any given cohomology class, can be understood completely by studying the spaces of all \sg\ 3- and 4-forms, \slc\ 3-forms and \slr\ 3-forms, respectively.  These spaces can be investigated using bundle-theoretic techniques such as characteristic classes and obstruction theory, and the purpose of this paper is to carry out such an investigation.

I begin by considering the existence of closed \sg\ 3- and 4-forms.  Since \(\svph_0\) and \(\svps_0\) both have stabiliser in \(\GL_+(7;\bb{R})\) isomorphic to \sg, the existence of either a \sg\ 3-form or a \sg\ 4-form is equivalent to a \sg-structure on \(\M\), i.e.\ a principal \sg-subbundle of the frame bundle of \(\M\).  In \cite{MAaSG2S}, L\^{e} proved that a closed, oriented 7-manifold admits \sg-structures \iff\ it is spin, and conjectured that the same result should hold without the assumption of compactness.  I begin by proving this conjecture:

\begin{Thm}\label{SG2-Exist}
Let \(\M\) be an oriented 7-manifold (not necessarily closed and possibly with boundary).  Then, \(\M\) admits \sg-structures \iff\ it is spin.
\end{Thm}

Combining \tref{SG2-Exist} with the homotopy equivalences established in \cite{RhPfCSF} yields the following corollary:
\begin{Thm}\label{SG2-Exist-Coh}
Let \(\M\) be an oriented 7-manifold.  If \(\M\) is spin, then every degree 3 cohomology class can be represented by a \sg\ 3-form and every degree
4 cohomology class can be represented by a \sg\ 4-form.
\end{Thm}

I then, investigate the link between closed \slc\ and \slr\ 3-forms in 6 dimensions and closed \sg\ 3-forms in 7 dimensions.  An \slc\ or \slr\ 3-form \(\rh\) on an oriented 6-manifold \(\N\) shall be termed extendible if there exists an oriented 7-manifold with boundary \(\M\) such that \(\del \M\) contains \(\N\) as a connected component and there exists a closed \sg\ 3-form \(\sph\) on \(\M\) such that \(\sph|_\N = \rh\).  Motivated by Donaldson's notion of \g-cobordism introduced in \cite{RoG2MwB}, two oriented 6-manifolds \((\N_1,\rh_1)\) and \((\N_2,\rh_2)\) equipped with closed, extendible, \slc\ (respectively \slr) 3-forms shall be termed \sg-cobordant if there exists an oriented 7-manifold \(\M\) with boundary \(\del\M = \N_1 \msm{\coprod} \ol{\N}_2\) and a closed \sg\ 3-form \(\sph\) on \(\M\) such that:
\ew
\sph|_{\N_1} = \rh_1 \et \sph|_{\N_2} = \rh_2
\eew
(where overline denotes orientation-reversal).

\begin{Thm}\label{htpy->cob}
Let \(\N\) be a 6-manifold and let \(\rh,\rh'\) be closed, extendible \slc\ (respectively \slr) 3-forms on \(\N\).  Suppose that \(\rh\) and \(\rh'\) are homotopic and lie in the same cohomology class. Then, \((\N,\rh)\) and \((\N,\rh')\) are \sg-cobordant.
\end{Thm}

I remark that, in contrast, the analogous result for \g-cobordisms is not known; see \cite{RoG2MwB}, particularly the discussion on p.\ 116.

The remainder of this paper investigates the topological properties of closed \slc\ and \slr\ 3-forms.  Firstly, it investigates when two closed \slc\ (respectively \slr) 3-forms are homotopic.  Secondly, motivated by \tref{htpy->cob}, it investigates when a given \slc\ (respectively \slr) 3-form is extendible.

Let \(\N\) be an oriented 6-manifold and let \(\mc{SL}_\bb{C}(\N)\) denote the set of homotopy classes of \slc\ 3-forms on \(\N\).  Since \(\SL(3;\bb{C})\) deformation retracts onto the simply connected subgroup \(\SU(3) \pc \SO(6)\), one may assign to each \(\SL(3;\bb{C})\) 3-form \(\rh\) a choice of spin structure on \(\N\), which depends only on the homotopy class of \(\rh\); see \sref{Htpic-SLC} for further details.  (Recall that, whilst a choice of spin structure technically depends on a choice of metric, this dependence is only superficial; see \cite[p.\ 82, Remark 1.9]{SG}).  Thus, writing \(\mc{S}{\kern-1pt}pin(\N)\) for the set of spin structures on \(\N\), there is a map:
\ew
\si: \mc{SL}_\bb{C}(\N) \to \mc{S}{\kern-1pt}pin(\N).
\eew

\begin{Thm}\label{SLC-htpy}
The map \(\si\) defines a bijective correspondence between homotopy classes of \slc\ 3-forms on \(\N\) (equivalently closed \slc\ 3-forms, or \slc\ 3-forms in any fixed degree 3 cohomology class) and spin structures on \(\N\).  In particular, \(\mc{SL}_\bb{C}(\N)\) is a torsor for the group \(\sch{1}{\N,\rqt{\bb{Z}}{2\bb{Z}}}\), i.e.\ \(\mc{SL}_\bb{C}(\N)\) admits a faithful, transitive action by \(\sch{1}{\N,\rqt{\bb{Z}}{2\bb{Z}}}\), and likewise for closed \slc\ 3-forms, or \slc\ 3-forms in a fixed cohomology class.
\end{Thm}

I remark that \tref{SLC-htpy} corrects an error in Donaldson's paper \cite[p.\ 116]{RoG2MwB}, where it is stated that any two \slc\ 3-forms on a given oriented 6-manifold are homotopic.

Regarding the extendibility of \slc\ 3-forms, this paper proves that:

\begin{Thm}\label{slc-ext-conds}
Let \(\N\) be an oriented 6-manifold.  If the Euler class \(e(\N) = 0\), then any \slc\ 3-form on \(\N\) is extendible.  In particular:
\begin{itemize}
\item If \(\N\) is open, then any \slc\ 3-form on \(\N\) is extendible.
\item If \(\N\) is closed and the Euler characteristic \(\ch(\N) = 0\), then any \slc\ 3-form on \(\N\) is extendible.
\end{itemize}
Conversely, if \(e(\N) \ne 0\) and in addition \(b^2 = 0\) (i.e.\ \(\sch{2}{\N;\bb{Z}}\) and \(\sch{4}{\N;\bb{Z}}\) are pure torsion), then no \slc\ 3-form on \(\N\) is extendible.
\end{Thm}

Turning to the case of \slr\ 3-forms, by using the results of Thomas \cite[Cor.\ 1.7]{FoTkPoM}, a lower bound on the number of homotopy classes of \slr\ 3-forms is obtained:

\begin{Thm}\label{slr-htpy}
Let \(\N\) be a closed, oriented, 6-manifold with \(e(\N) = 0\) and suppose \(w_2(\N)^2 = 0\).  Write \(\rh_2:\sch{4}{\N;\bb{Z}} \to \sch{4}{\N;\rqt{\bb{Z}}{2\bb{Z}}}\) for reduction modulo 2 and define:
\ew
\sch{4}{\N;\bb{Z}}_{\bot w_2} = \lt\{ u \in \sch{4}{\N;\bb{Z}} ~\m|~ \rh_2u \cup w_2(\N) = 0 \rt\}.
\eew
Then, there is an injection from \(\rqt{\sch{4}{\N;\bb{Z}}_{\bot w_2}}{\text{2-torsion}}\) into the set of homotopy classes of \slr\ 3-forms on \(\N\) (equivalently closed \slr\ 3-forms, or \slr\ 3-forms in any fixed degree 3 cohomology class).  In particular, if \(\N\) is spin and \(b^4(\N)>0\), then each of these sets is infinite.
\end{Thm}

As an immediate corollary of \tref{slr-htpy}, one obtains:
\begin{Cor}\label{spin-cor}
Let \(\N\) be a closed, oriented, spin 6-manifold.  Then, \(\N\) admits \slr\ 3-forms \iff\ \(e(\N) = 0\).
\end{Cor}

The paper ends by investigating the extendibility of \slr\ 3-forms.  In \sref{htpic-slr}, infinitely many distinct homotopy classes of (closed) non-extendible \slr\ 3-forms are shown to exist on the manifold \(\bb{T}^2 \x \mc{E}\), where \(\mc{E}\) denotes the Enriques surface, while in \sref{ext-slr}, 652 distinct homotopy classes of (closed) extendible \slr\ 3-forms are constructed on the manifold \(\bb{T}^6\).

The results of the present paper were obtained during the author's doctoral studies, which where supported by EPSRC Studentship 2261110.\\

\section{Preliminaries}

\subsection{Elementary properties of \(\mb{\tld{\Gg}_2}\), \(\mb{\SL(3;\bb{C})}\) and \(\mb{\SL(3;\bb{R})^2}\) forms}

Recall from \cite[\S2, Thm.\ 2]{MwEH} that the stabiliser \sg\ of \(\svph_0\) in \(\GL_+(7;\bb{R})\) lies in \(\SO(3;4)\).  In particular, \sg\ preserves the metric:
\ew
\tld{g}_0 = \lt(\th^1\rt)^{\ts2} + \lt(\th^2\rt)^{\ts2} + \lt(\th^3\rt)^{\ts2} - \lt(\th^4\rt)^{\ts2} - \lt(\th^5\rt)^{\ts2} - \lt(\th^6\rt)^{\ts2} - \lt(\th^7\rt)^{\ts2}.
\eew
(Indeed, one may compute that, writing \(\tld{\Hs}_0\) for the Hodge star induced by \(\tld{g}_0\) and the usual orientation on \(\bb{R}^7\), the identity \(\svps_0 = \tld{\Hs}_0\svph_0\) holds; see \cite[Ch.\ 1.C]{EM} for an exposition of the elementary properties of pseudo-Riemannian metrics: it is this identity which underpins the equivalence between \sg\ 3-forms and \sg\ 4-forms.)  It follows that given an oriented 7-manifold \(\M\), any \sg\ 3-form \(\sph\) on \(\M\) induces a pseudo-Riemannian metric of signature \((3,4)\) on \(\M\), denoted \(g_{\sph}\).  Now, let \(x \in \M\) and fix a hyperplane \(\bb{B} \pc \bb{R}^7\).  Since \(g_{\sph}|_x\) is non-degenerate, the space:
\ew
\bb{B}^\bot = \lt\{ u \in \T_x \M ~\m|~ \text{for all } w \in \bb{B}: g_{\sph}|_x(u,w) = 0\rt\}
\eew
is a 1-dimensional subspace of \(\T_x \M\) such that \(\lt(\bb{B}^\bot\rt)^\bot = \bb{B}\).  Let \(u\) be any non-zero element of \(\bb{B}^\bot\).  Since \(g_{\sph}\) is split-signature, it follows that \(g_{\sph}|_x(u,u)\) may be either positive, negative, or zero (a condition which is independent of choice of \(u \in \bb{B}^\bot \osr \{0\}\)).  Say that \(\bb{B}\) is spacelike, timelike, or null according to whether \(g_{\sph}|_x(u,u)\) is positive, negative or zero.

This paper shall consider the cases where \(\bb{B}\) is either spacelike or timelike.  In these cases, one can write \(\T_x \M = \bb{B}^\bot \ds \bb{B}\).  Given a choice of orientation of \(\bb{B}\) (equivalently, a choice of orientation of the 1-dimensional subspace \(\Ann(\bb{B}) \pc \T_x^* \M\)), one may choose a correctly oriented generator \(\th\) of \(\Ann(\bb{B})\) and write:
\ew
\sph|_x = \th \w \om + \rh,
\eew
where \(\om \in \ww{2}\bb{B}^*\), \(\rh \in \ww{3}\bb{B}^*\), and the embedding \(\ww{p}\bb{B}^* \emb \ww{p}\T^*_x \M\) is defined by declaring \(u \hk \al = 0\) for all \(u \in \bb{B}^\bot\) and \(\al \in \ww{p}\bb{B}^*\).  Then, (see \cite[Prop.\ 9.6]{RhPfCSF}) \(\rh\) is either an \slc\ 3-form or an \slr\ 3-form, depending on whether \(\bb{B}\) is spacelike or timelike.  Moreover, the normal component \(\om\) can be characterised in terms of the algebraic structures induced by \(\rh\), as I now describe.

Begin with the spacelike case.  Since \(\SL(3;\bb{C}) \pc \GL(3;\bb{C})\), \(\rh\) defines a complex structure \(J_\rh\) on \(\bb{B}\); explicitly, in the case of \(\rh_- \in \ww{3}\lt(\bb{R}^6\rt)^*\), \(J_{\rh_-}\) is the standard \acs\ defined by declaring \(\th^1 + i \th^2\), \(\th^3 + i \th^4\) and \(\th^5 + i \th^6\) to be a basis of \((1,0)\)-forms on \(\lt(\bb{R}^6,J_{\rh_-}\rt)\).  Thus, \(\rh\) also defines a map:
\e\label{calJrh}
\bcd[row sep = 0pt]
\cal{J}_\rh:\ww{2}\bb{B}^* \ar[r] & \ss{2}\bb{B}^*\\
\om \ar[r, maps to] & \Big\{(a,b) \mt -\frac{1}{2}\lt[\om(J_\rh a,b) + \om(J_\rh b,a)\rt]\Big\}.
\ecd
\ee
The following was proven by the author in \cite{RhPfCSF}:
\begin{Prop}[{\cite[Prop.\ 9.24]{RhPfCSF}}]\label{SRP}
Let \(\rh\) be an \(\SL(3;\bb{C})\) 3-form on \(\bb{B}\).  Then, \(\th \w \om + \rh\) is a \sg\ 3-form on \(\T_x \M\) \iff\ the symmetric bilinear form \(\cal{J}_\rh \om\) on \(\bb{B}\) has signature \((2,4)\).
\end{Prop}

Now, turn to the timelike case.  Recall that a para-complex structure on a real vector space is an endomorphism \(I\) such that \(I^2 = \Id\) and the \(\pm1\)-eigenspaces of \(I\) have equal dimension.  Since \slr\ preserves the para-complex structure \(I_0\) on \(\bb{R}^6\) defined by:
\e\label{I0}
I_0(e_i) = \begin{cases*} e_i & i = 1,2,3\\ -e_i & i = 4,5,6, \end{cases*}
\ee
the \slr\ 3-form \(\rh\) induces a para-complex structure \(I_\rh\) on \(\bb{B}\) and hence induces a map:
\e\label{calIrh}
\bcd[row sep = 0pt]
\cal{I}_\rh:\ww{2}\bb{B}^* \ar[r] & \ss{2}\bb{B}^*\\
\om \ar[r, maps to] & \Big\{(a,b) \mt \frac{1}{2}\lt[\om(I_\rh a,b) + \om(I_\rh b,a)\rt]\Big\}.
\ecd
\ee
The following was also proven by the author in \cite{RhPfCSF}:
\begin{Prop}[{\cite[Prop.\ 9.15]{RhPfCSF}}]\label{TRP}
Let \(\rh\) be an \slr\ 3-form on \(\bb{B}\).  Then, \(\th \w \om + \rh\) is a \sg\ 3-form on \(\T_x \M\) \iff\ the symmetric bilinear form \(\cal{I}_\rh \om\) on \(\bb{B}\) has signature \((3,3)\) and \(\om^3 < 0\).
\end{Prop}

I end this subsection with one final remark.  Given an \slr\ 3-form \(\rh\) on \(\N\), write \(E_{\pm,\rh}\) for the \(\pm1\) eigenbundles of the para-complex structure \(I_\rh\) on \(\ts \N\) induced by \(\rh\).  Then, \(\rh|_{E_{\pm,\rh}} \ne 0\) for both \(\pm1\), and thus \(\rh\) defines a canonical orientation on \(E_{\pm,\rh}\); in particular, one may regard \(E_{\pm,\rh}\) as sections of the oriented Grassmannian bundle \(\oGr_3(\T\N)\).\\

\subsection{A vanishing result for natural cohomology classes}

The aim of this subsection is to prove the following result:
\begin{Lem}\label{closed-exhaust}
Suppose there exists an assignment to each \(n\)-manifold \(\M\) (with, possibly empty, boundary) of a degree \(p\) cohomology class \(\nu(\M) \in \sch{p}{\M;G}\), where \(G\) is either a field or a finite Abelian group, which is natural in the sense that for each embedding \(f:\M \emb \M'\) of \(n\)-manifolds with boundary:
\ew
\nu(\M) = f^*\nu(\M').
\eew
Then, if \(\nu\) vanishes on every closed (respectively closed, oriented) \(n\)-manifold, it vanishes on every (respectively every oriented) \(n\)-manifold with boundary.
\end{Lem}

Examples of such classes \(\nu\) are any cohomology class which is constructed only from Stiefel--Whitney classes, or only from the reduction of the Chern, Pontrjagin and Euler classes to real coefficients.  More generally, for any cohomology operation \(\Th:\sch{p}{-;G} \to \sch{q}{-;G'}\) (see \cite[p.\ 448]{AT}), if \(\nu(\M) \in \sch{p}{\M;G}\) is natural, then \(\Th \circ \nu(\M) \in \sch{q}{\M;G'}\) is also natural.  I remark also that, whilst only the case \(G = \rqt{\bb{Z}}{2\bb{Z}}\) will be used in this paper, more general \(G\) is allowed in \lref{closed-exhaust}, since the proof for all such \(G\) is essentially the same.

\begin{proof}[Proof of \lref{closed-exhaust}]
By assumption \(\nu(\M) = 0\) for all closed (respectively closed, oriented) \(n\)-manifolds \(\M\).  The proof proceeds by considering three cases of increasing generality.\\

{\bf\boldmath Case 1: \(\M\) is compact with boundary.}  Consider the double \(\cal{D}\M = \M \cup_{\del\M} \ol{\M}\) formed by gluing \(\M\) to a second copy of itself \(\ol{\M}\) (now with the opposite orientation, if appropriate) along the boundary \(\del\M\). Then, \(\cal{D}\M\) is a closed (respectively closed, oriented) \(n\)-manifold and thus \(\nu(\cal{D}\M) = 0\), by assumption. Writing \(\io:\M\emb\cal{D}\M\) for the natural inclusion, the naturality of \(\nu\) implies that:
\ew
\nu(\M) = \io^*\nu(\cal{D}\M) = 0,
\eew
as claimed.\\

{\bf\boldmath Case 2: \(\M\) is non-compact and without boundary.}  Let \(f:\M \to \bb{R}\) be a proper Morse function (see, e.g.\ \cite[Cor.\ 6.7]{MT}) and choose increasing unbounded sequences \(i_k\in\bb{R}_{>0}\) and \(j_k \in \bb{R}_{>0}\) such that both \(i_k\) and \(-j_k\) are regular values of \(f\) for all \(k\in\bb{N}\). Then, for each \(k\) the subset \(f^{-1}[-j_k,i_k] = \M_k\) is a compact submanifold with boundary of \(\M\) (see \cite[Lem., p.\ 62]{DT} for a similar result).  Moreover, each \(\M_{k+1}\) is obtained from \(\M_k\) by attaching a finite number of \(m\)-cells, for suitable choices of \(m\) and thus the function \(f\) gives \(\M\) the structure of a CW complex such that each \(\M_k\) is a subcomplex of \(\M\).  Define:
\ew
\IL \sch{p}{\M_k; G} = \lt\{(m_k)_k \in \prod_{i=0}^\infty \sch{p}{\M_k; G} ~\m|~\text{for all }k\ge0: m_{k+1}|_{\M_k} = m_k\rt\}.
\eew
Suppose initially that \(G = \bb{Q}\), or \(G = \rqt{\bb{Z}}{q\bb{Z}}\) for some prime \(q\).  Then, by \cite[Prop.\ 3F.5]{AT}, the natural map:
\e\label{IL-iso}
\sch{p}{\M; G} &\to \IL \sch{p}{\M_k; G}\\
m &\mt \lt(m|_{\M_k}\rt)_k
\ee
is an isomorphism.  Thus, \(\nu(\M) = 0\) \iff\ \(\nu(\M)|_{\M_k} = 0\) for each \(k\).  However, by naturality \(\nu(\M)|_{\M_k} = \nu(\M_k)\), which vanishes by Case 1, yielding \(\nu(\M) = 0\), as required.

For more general \(G\) (see \cite[Lem.\ 2]{OAHT}), \eref{IL-iso} is replaced by the short exact sequence:
\ew
0 \to \IL{}^1 \sch{p-1}{\M_k; G} \to \sch{p}{\M; G} &\to \IL \sch{p}{\M_k; G} \to 0,
\eew
where \(\IL{}^1\) is the first right-derived functor of \(\IL\) (see \cite{OAHT} for a more explicit definition).  Thus, to prove the lemma when \(G\) is an arbitrary field or finite Abelian group, it suffices to prove that \(\IL{}^1\sch{p-1}{\M_k; G} = 0\) in this case.  However this is clear: since each \(\M_k\) is a finite cell complex, the spaces \(\sch{p-1}{\M_k; G}\) are finite-dimensional \(G\) vector spaces if \(G\) is a field, and are finite Abelian groups if \(G\) is a finite Abelian group.  The result now follows by \cite[Exercise 3.5.2]{AItHA}.\\

{\bf\boldmath Case 3: \(\M\) is non-compact with boundary.}  By considering the double \(\mc{D}\M\) of \(\M\) and using case 2, it follows that \(\nu(\M) = 0\).

\end{proof}

\section{Existence of \sg-structures}

The aim of this section is to prove \tref{SG2-Exist}.  Recall the following definition, taken from \cite{CG}:

\begin{Defn}
Recall that a 3-form \(\ph\) on \(\bb{R}^7\) is called a \g\ 3-form if there is an orientation-preserving endomorphism of \(\bb{R}^7\) identifying \(\ph\) with the 3-form:
\ew
\vph_0 = \th^{123} + \th^{145} + \th^{167} + \th^{246} - \th^{257} - \th^{347} - \th^{356} \in\ww{3}\lt(\bb{R}^7\rt)^*.
\eew
The stabiliser of \(\vph_0\) in \(\GL_+(7;\bb{R})\) is \(\Gg_2 \pc \SO(7)\) which preserves the standard Euclidean inner product on \(\bb{R}^7\) and hence \(\ph\) induces an inner product on \(\bb{R}^n\), denoted \(g_\ph\).  An oriented 3-plane \(C \in \oGr_3(\bb{R}^7)\) is then called calibrated \wrt\ \(\ph\) if, writing \(vol_C\) for the volume form on \(C\) induced by the metric \(g_\ph|_C\) and the orientation on \(C\), one has:
\ew
\ph|_C = vol_C.
\eew
\end{Defn}

Analogously, let \(\sph\) be a \sg\ 3-form on \(\bb{R}^7\).  I term an oriented 3-plane \(C \in \oGr_3(\bb{R}^7)\) positively calibrated if \(g_{\sph}\) is positive definite on \(C\) and, writing \(vol_C\) for the volume form on \(C\) induced by the metric \(g_{\sph}|_C\) and the orientation on \(C\), one has:
\ew
\sph|_C = vol_C.
\eew
It is well-known that \g\ acts transitively on the set of calibrated 3-planes and that the stabiliser of any calibrated 3-plane is isomorphic to \(\SO(4)\) (see \cite[\S10.8]{CMwSH}).  Similarly:
\begin{Prop}
\sg\ acts transitively on the set of positively calibrated 3-planes and the stabiliser of any positively calibrated 3-plane is a maximal compact subgroup of \sg, isomorphic to \(\SO(4)\).
\end{Prop}

\begin{proof}
To prove transitivity of the action, consider the standard \sg\ 3-form \(\svph_0\) on \(\bb{R}^7\) and let \(C\) be positively calibrated \wrt\ \(\svph_0\).  Pick an oriented orthonormal basis \((c_1,c_2,c_3)\) of \(C\) \wrt\ \(\tld{g}_0|_C\) (which exists since \(\tld{g}_0\) is positive definite on \(C\)).  By \cite[Prop.\ 2.3]{SG2SoPRM}, \sg\ acts transitively on ordered pairs of orthonormal, spacelike vectors in \(\bb{R}^7\), so \wlg\ \(c_i = e_i\) (\(i=1,2\)).  Since \(\svph_0|_C = vol_C\) and \((c_1,c_2,c_3)\) is an oriented orthonormal basis of \(C\), one has:
\ew
\svph_0(c_1,c_2,c_3) = 1.
\eew
It follows that \(c_3 = e_3 + u\) for some \(u \in \<e_4,...,e_7\?\).  Since \(\tld{g}_0(c_3,c_3) = 1\), \(u\) satisfies \(\tld{g}_0(u,u) = 0\) and hence \(u = 0\), since \(\tld{g}_0\) is negative definite on \(\<e_4,...,e_7\?\).  Thus, \(C = \<e_1,e_2,e_3\?\) up to the action of \sg\ and hence \sg\ acts transitively on positively calibrated 3-planes.  The statement regarding stabilisers is proven in \cite[Prop.\ 4.4]{A3F&ESLGoTG2}.

\end{proof}

(Positively) calibrated 3-planes have the following desirable property:
\begin{Lem}\label{Cal-PCal}
(1) Let \(\ph\) be a \g\ 3-form on \(\bb{R}^7\) and let \(C\) be a calibrated 3-plane. Then:
\ew
\ph_C = 2\ph|_C - \ph
\eew
defines a \sg\ 3-form on \(\bb{R}^7\) and \(C\) is positively calibrated \wrt\ \(\ph_C\) (here \(\ph|_C\) is interpreted as a 3-form on \(\bb{R}^7\) using the splitting \(\bb{R}^7 = C \ds C^\bot\), where the orthocomplement is taken \wrt\ \(g_\ph\)).\\

(2) Let \(\sph\) be a \sg\ 3-form on \(\bb{R}^7\) and let \(C\) be a positively calibrated 3-plane. Then:
\ew
\sph_C = 2\sph|_C - \sph
\eew
defines a \g\ 3-form on \(\bb{R}^7\) and \(C\) is calibrated \wrt\ \(\sph_C\) (again \(\sph|_C\) is interpreted as a 3-form on \(\bb{R}^7\) using the splitting \(\bb{R}^7 = C \ds C^\bot\), where the orthocomplement is taken \wrt\ \(g_{\sph}\)).
\end{Lem}

\begin{proof}
The proof is by direct calculation.  For (1), since \g\ acts transitively on the set of calibrated 3-planes, \wlg\ one may assume that \(\ph = \vph_0\) and \(C = \<e_1,e_2,e_3\?\).  Then:
\ew
\ph_C = 2\th^{123} - \lt(\th^{123} + \th^{145} + \th^{167} + \th^{246} - \th^{257} - \th^{347} - \th^{356}\rt) = \th^{123} - \th^{145} - \th^{167} - \th^{246} + \th^{257} + \th^{347} + \th^{356}
\eew
which can be seen to be of \sg-type by considering the orientation-preserving automorphism \(\Id \ds -\Id\) on \(\<e_1\? \ds \<e_2,...,e_7\? \cong \bb{R}^7\), and which has \(C\) as a positively calibrated 3-plane.  The converse is similar.

\end{proof}

Since \(\SO(4) \pc \tld{\Gg}_2\) is a maximal compact subgroup, the quotient \(\rqt{\tld{\Gg}_2}{\SO(4)}\) is contractible.  Thus, given any oriented 7-manifold \(\M\) equipped with a \sg\ 3-form \(\sph\), there exists a rank 3 distribution \(C\) on \(\M\) which is pointwise positively calibrated \wrt\ \(\sph\).  The corresponding result in the \g\ case is non-trivial, since \(\rqt{\Gg_2}{\SO(4)}\) is not contractible.
\begin{Prop}\label{Cal-Exist}
Let \(\M\) be an oriented 7-manifold and let \(\ph\) be a \g\ 3-form on \(\M\).  Then, \(\M\) admits a rank 3 distribution \(C\) which is pointwise calibrated \wrt\ \(\ph\).
\end{Prop}

\begin{proof}
The proof is a generalisation of Friedrich--Kath--Moroianu--Semmelmann's proof of the existence of \(\SU(2)\)-structures on closed 7-manifolds with \g-structures (see \cite[Thm.\ 3.2]{ONPG2S}).  Define a cross-product \(\x\) on \(\M\) by the equation:
\ew
g(u_1 \x u_2, u_3) = \ph(u_1,u_2,u_3)
\eew
for all \(p \in \M\) and \(u_i \in \T_p\M\), (\(i=1,2,3\)).  Direct calculation shows that if \(u_1\) and \(u_2\) are linearly independent, then \(\text{Span}\<u_1,u_2,u_1\x u_2\?\), together with its natural orientation induced by the above ordering of basis vectors, defines a calibrated 3-plane in \(\T_p\M\).  If \(\M\) is closed, \(\M\) admits a pair of everywhere linearly independent vector fields by \cite[Table 2]{VFoM}.

Therefore, to prove \pref{Cal-Exist} for open manifolds (i.e.\ for non-compact manifolds and manifolds with boundary) it suffices to prove that every open orientable 7-manifold \(\M\) also admits two everywhere linearly independent vector fields.  By \cite[Thm.\ 12.1]{CC} (see also the preceding discussion {\it op.\ cit.}), the condition \(w_6(\M)=0\) is necessary and sufficient to ensure the existence of two vector fields \(X\) and \(Y\) defined over the 6-skeleton of \(\M\) which are everywhere linearly independent. Moreover, since \(\M\) is open, \(\M\) deformation retracts onto a subcomplex of its 6-skeleton (cf.\ \cite[Prop.\ 4.3.1]{ItthP}) and thus \(\M\) itself admits two globally defined vector fields \(X\) and \(Y\) \iff\ \(w_6(\M) = 0\).  By \cite[Thm.\ III]{OtSWCoaM}, \(w_6\) vanishes on every closed oriented 7-manifold.  Thus, by \lref{closed-exhaust} it follows that \(w_6\) vanishes on every oriented 7-manifold, completing the proof.

\end{proof}

Using \pref{Cal-Exist}, I now prove \tref{SG2-Exist}:

\begin{proof}[Proof of \tref{SG2-Exist}]
By a well-known result of Gray (\cite[Remark 3]{RoG2S}; cf.\ \cite{VCPoM}), an oriented 7-manifold \(\M\) admits a \g-structure \(\ph\) \iff\ it is spin.  By \pref{Cal-Exist}, \(\M\) admits a pair \((\ph,C)\) of a \g\ 3-form \(\ph\) together with a rank 3 distribution \(C\), pointwise calibrated \wrt\ \(\ph\), \iff\ \(\M\) admits a \g-structure.  By \lref{Cal-PCal}, \(\M\) admits a pair \(\lt(\sph,C\rt)\) with \(\sph\) a \sg\ 3-form and \(C\) a rank 3 distribution pointwise positively calibrated \wrt\ \(\sph\), \iff\ \(\M\) admits a pair \((\ph,C)\) with \(\ph\) a \g\ 3-form and \(C\) a rank 3 distribution pointwise calibrated \wrt\ \(\ph\).  Finally -- as discussed above -- since \(\SO(4) \pc \tld{G}_2\) is maximal compact, the quotient space \(\rqt{\tld{G}_2}{\SO(4)}\) is contractible and thus a manifold \(\M\) admits a \sg\ 3-form \(\sph\) \iff\ it admits a pair \(\lt(\sph,C\rt)\) with \(\sph\) a \sg\ 3-form and \(C\) a rank 3 distribution pointwise positively calibrated \wrt\ \(\sph\). The result follows by combining these four logical equivalences.

\end{proof}
~

\section{\sg-Cobordisms}

The aim of this section is to prove \tref{htpy->cob}.  For brevity of notation, the term `\(\SL\) form' shall be used to denote either an \(\SL(3;\bb{C})\) 3-form or an \(\SL(3;\bb{R})^2\) 3-form, as appropriate.

\begin{Defn}\label{ext-def}
Let \(\N\) be an oriented 6-manifold.  Let \(\ww{\bin{2}{3}}\T^*\N\) denote the pullback of the bundle \(\ww{2}\T^*\N \to \N\) along the bundle \(\ww{3}\T^*\N \to \N\) and write \(\fr{p}:\ww{\bin{2}{3}}\T^*\N \to \ww{2}\T^*\N\) for the natural projection:
\ew
\bcd
\ww{\bin{2}{3}}\T^*\N \ar[d] \ar[r, "\fr{p}"] & \ww{2}\T^*\N \ar[d]\\
\ww{3}\T^*\N \ar[r] & \N
\ecd
\eew
Write \(\ww[+]{3}\T^*\N\) for the bundle of \slr\ 3-forms on \(\N\) and define a subbundle \(\mc{E}_+\) of \(\ww{\bin{2}{3}}\T^*\N\) over \(\ww[+]{3}\T^*\N\) by declaring the fibre over \(\rh \in \ww[+]{3}\T^*\N\) to be:
\ew
\mc{E}_+|_\rh = \lt\{ \om \in \ww{\bin{2}{3}}\T^*\N|_\rh ~\m|~ \cal{I}_\rh (\fr{p}\om) \text{ has signature } (3,3) \text{ and } (\fr{p}\om)^3 < 0 \rt\}
\eew
(see \eref{calIrh} for the definition of \(\cal{I}_\rh\)).  Likewise, write \(\ww[-]{3}\T^*\N\) for the bundle of \slc\ 3-forms on \(\N\) and define a subbundle \(\mc{E}_-\) of \(\ww{\bin{2}{3}}\T^*\N\) over \(\ww[-]{3}\T^*\N\) by:
\ew
\mc{E}_-|_\rh = \lt\{ \om \in \ww{\bin{2}{3}}\T^*\N|_\rh ~\m|~ \cal{J}_\rh (\fr{p}\om) \text{ has signature } (2,4) \rt\}
\eew
(see \eref{calJrh} for the definition of \(\cal{J}_\rh\)).

Now, let \(\rh\) be an \(\SL\) form, i.e.\ a section of \(\ww[\pm]{3}\T^*\N\), as appropriate.  I term \(\rh\) extendible if there exists a lift \(\om\) of the section \(\rh\) along the map \(\mc{E}_\pm \to \ww[\pm]{3}\T^*\N\):
\ew
\bcd
& \mc{E}_\pm \ar[d]\\
\N \ar[ur, dashed, "\om"] \ar[r, ^^22 \rh ^^22] & \ww[\pm]{3}\T^*\M
\ecd
\eew
\end{Defn}

\begin{Prop}\label{ext-off}
Let \(\N\) be an oriented 6-manifold and let \(\rh\) be a (closed) \(\SL\) form on \(\N\).  Then, the following are equivalent:
\begin{itemize}
\item There exists an oriented 7-manifold with boundary \(\M\) together with a (closed) \sg\ 3-form \(\sph\) such that \(\N\) is a connected component of \(\del\M\) and \(\sph|_\N = \rh\);
\item \(\rh\) is extendible.
\end{itemize}
\end{Prop}

\begin{proof}
Suppose that \(\rh\) is extendible and let \(\om\) be a lift of \(\rh\) along \(\mc{E}_\pm \to \ww[\pm]{3}\T^*\N\).  Let \(f:\N \to (0,\infty)\) be chosen later and consider the manifold:
\ew
\M = \lt\{ (t,p) \in [0,\infty) \x \N ~\m|~ 0 \le t < f(p) \rt\}.
\eew
Writing \(\pi:\M \to \N\) for the natural projection, define a 3-form \(\sph\) on \(\M\) via:
\ew
\sph = \dd t \w \pi^*(\fr{p}\om) + \pi^*\rh + t\pi^*\dd(\fr{p}\om).
\eew
By \prefs{SRP} and \ref{TRP}, \(\sph\) is of \sg-type along \(\{0\} \x \N\) and hence, by the stability of \sg\ 3-forms, it is of \sg-type on all of \(\M\) if \(f(p)\) is sufficiently small, depending on \(p\in \N\).  Moreover, \(\dd\sph = \pi^*\dd\rh\) and thus if \(\rh\) is closed on \(\N\), then \(\sph\) is closed on \(\M\), as claimed.

Conversely if \(\N\) is a connected component of \(\del\M\), then by the Collar Neighbourhood Theorem (\cite[Thm.\ 9.25]{ItSM}; cf.\ also \cite[Lem.\ 5]{LFIoTM}) there is an open neighbourhood of \(\N\) in \(\M\) which is diffeomorphic to \([0,1) \x \N\).  One now simply applies the above argument in reverse.

\end{proof}

I next prove \tref{htpy->cob}.  Recall the statement of the theorem:\vs{3mm}

\noindent{\bf Theorem \ref{htpy->cob}.}
\em Let \(\N\) be a 6-manifold and let \(\rh,\rh'\) be closed, extendible \(\SL\) forms on \(\N\).  Suppose that \(\rh\) and \(\rh'\) are homotopic and lie in the same cohomology class. Then, \((\N,\rh)\) and \((\N,\rh')\) are \sg-cobordant.\vs{2mm}\em

\begin{proof}
Let \(\rh_t\) denote a homotopy of sections of \(\ww[\pm]{3}\T^*\N\) over \(\N\) such that \(\rh_0 = \rh\) and \(\rh_1 = \rh'\), and choose a lift \(\om\) of \(\rh\) along \(\mc{E}_\pm \to \ww[\pm]{3}\T^*\N\).  Using the covering homotopy property for fibre bundles \cite[Ch.\ III, Thm.\ 4.1]{HT}, there is a homotopy of sections \(\om_t:\N\to\mc{E}_\pm\) such that for each \(t\in[0,1]\), \(\om_t\) is a lift of \(\rh_t\) along \(\mc{E}_\pm \to \ww[\pm]{3}\T^*\N\).

Now, consider the space \(\M = [0,1]_t\x\N\) and write \(\pi:\M \to \N\) for the natural projection.  The families \(\fr{p}\om_t\), \(\dd_\N\fr{p}\om_t\) and \(\rh_t\) (where \(\dd_\N\) denotes the usual exterior derivative on \(\N\)) naturally define sections of the bundles \(\pi^*\ww{2}\T^*\N\) and \(\pi^*\ww{3}\T^*\N\) over \(\M\), as appropriate.  Let \(f_1:\M \to [0,\infty)\) be a smooth function which is identically 1 on an open neighbourhood of \(\{0\} \x \N\), but which vanishes outside some larger neighbourhood of \(\{0\} \x \N\).  Likewise, let \(f_2:\M \to [0,\infty)\) be identically 1 on a small open neighbourhood of \(\{1\} \x \N\) and identically 0 outside some larger open neighbourhood.  Define a 3-form \(\sph\) on \(\M\) via:
\ew
\sph = \dd t \w (\fr{p}\om_t) + \rh_t + [t f_1 + (t-1)f_2] \dd_\N(\fr{p}\om_t).
\eew
Then, by \prefs{SRP} and \ref{TRP}, \(\dd t \w (\fr{p}\om_t) + \rh_t\) is of \sg-type on \(\M\); hence, by the stability of \sg\ 3-forms, if the supports of \(f_1\) and \(f_2\) are chosen to be sufficiently small, then \(\sph\) too is of \sg-type.  Moreover, writing \(\al\) for the common cohomology class defined by \(\rh\) and \(\rh'\) and identifying the cohomology rings of \(\N\) and \(\M\) as usual, a direct calculation shows that \(\dd\sph = 0\) on \(\{f_1 \equiv 1\} \cup \{f_2 \equiv 1\}\) and that \(\sph\) represents the restriction of the class \(\al\) to \(\{f_1 \equiv 1\} \cup \{f_2 \equiv 1\}\).  Let \(\Op(\del \M) \cc \{f_1 \equiv 1\} \cup \{f_2 \equiv 1\}\) denote a small open neighbourhood of \(\del\M\) in \(\M\) and define:
\caw
\Om^3_{\svph_0}\lt(\M;\sph|_{\Op(\del \M)}\rt) = \lt\{ \sph' \in \Om^3_{\svph_0}(\M) ~\m|~ \sph'|_{\Op(\del\M)} = \sph|_{\Op(\del\M)}\rt\};\\
\CL^3_{\svph_0}\lt(\al;\sph|_{\Op(\del \M)}\rt) = \lt\{ \sph' \in \Om^3_{\svph_0}(\M;\sph'|_{\Op(\del \M)}) ~\m|~ \dd\sph' = 0 \text{ and } [\sph'] = \al \in \dR{3}(\M)\rt\}.
\caaw
Then, \(\sph\) defines an element of \(\Om^3_{\svph_0}\lt(\M;\sph|_{\Op(\del \M)}\rt)\).  It was proven by the author in \cite[Thm.\ 9.44]{RhPfCSF} that \(\CL^3_{\svph_0}\lt(\al;\sph|_{\Op(\del \M)}\rt) \emb \Om^3_{\svph_0}\lt(\M;\sph|_{\Op(\del \M)}\rt)\) is a homotopy equivalence.  Thus, one can deform \(\sph \in \Om^p_{\svph_0}\lt(\M;\sph|_{\Op(\del \M)}\rt)\) relative to \(\Op(\del \M)\) into a closed \sg\ 3-form \(\sph'\) on \(\M\) (representing the class \(\al\)). The pair \((\M,\sph')\) then gives the required cobordism from \((\N,\rh)\) to \((\N,\rh')\).

\end{proof}
~

\section{Topological properties of \slc\ 3-forms}

The aim of this section is to investigate when two \slc\ 3-forms are homotopic, and when a single \slc\ 3-form is extendible.\\

\subsection{\boldmath Homotopic \(\SL(3;\bb{C})\) 3-forms}\label{Htpic-SLC}

In this subsection, I prove \tref{SLC-htpy}.  Recall the statement of the theorem:\vs{3mm}

\noindent{\bf Theorem \ref{SLC-htpy}.}
\em There is a bijective correspondence between homotopy classes of \slc\ 3-forms on \(\N\) (equivalently closed \slc\ 3-forms, or \slc\ 3-forms in any fixed degree 3 cohomology class) and spin structures on \(\N\).  In particular, \(\mc{SL}_\bb{C}(\N)\) is a torsor for the group \(\sch{1}{\N,\rqt{\bb{Z}}{2\bb{Z}}}\), and likewise for closed \slc\ 3-forms, or \slc\ 3-forms in a fixed cohomology class.
\vs{2mm}\em

I begin by remarking that the existence of the faithful, transitive action of \(\sch{1}{\N; \rqt{\bb{Z}}{2\bb{Z}}}\) on \(\mc{SL}_\bb{C}(\N)\) can be proven directly via classical Obstruction Theory, without any reference to spin structures.  Indeed, the fibre of the bundle \(\ww[-]{3}\T^*\N\) of \slc\ 3-forms over \(\N\) is homeomorphic to \(\rqt{\GL_+(6;\bb{R})}{\SL(3;\bb{C})}\), which deformation retracts onto the space \(\rqt{\SO(6)}{\SU(3)} \cong \bb{RP}^7\).  Since \(\pi_n\lt(\bb{RP}^7\rt) = 0\) for \(n = 2,...,6\), classical Obstruction Theory (see \cite[Thm.\ 6.13]{EoHT}) implies that the set of homotopy classes of sections of \(\ww[-]{3}\T^*\N\) over \(\N\) is a torsor for the group \(\sch{1}{\N;\pi_1\lt(\ww[-]{3}\T^*_\pt\N\rt)}\), where \(\pi_1\lt(\ww[-]{3}\T^*_\pt\N\rt)\) denotes the bundle of groups over \(\N\) given by the first fundamental groups of the fibres of \(\ww[-]{3}\T^*\N\).  Since \(\pi_1\lt(\bb{RP}^7\rt) \cong \rqt{\bb{Z}}{2\bb{Z}}\) has no non-trivial automorphisms, the bundle \(\pi_1\lt(\ww[-]{3}\T^*_\pt\N\rt)\) itself must be trivial (or simple, in the terminology of \cite{EoHT}; see p.\ 263 {\it op.\ cit.}) and thus \(\sch{1}{\N;\pi_1\lt(\ww[-]{3}\T^*_\pt\N\rt)}\) is simply the usual cohomology group \(\sch{1}{\N;\rqt{\bb{Z}}{2\bb{Z}}}\).  Thus, the set of homotopy classes of \slc\ 3-forms is a torsor for \(\sch{1}{\N;\rqt{\bb{Z}}{2\bb{Z}}}\), as claimed.

The action of \(\ch \in \sch{1}{\N;\rqt{\bb{Z}}{2\bb{Z}}}\) on \(\mc{SL}_\bb{C}(\N)\) admits a very explicit description in the case that \(\ch\) lies in the image of the natural map \(r_2:\sch{1}{\N;\bb{Z}} \to \sch{1}{\N;\rqt{\bb{Z}}{2\bb{Z}}}\).  Indeed, firstly note that, given any \(\rh \in \ww[-]{3}\lt(\bb{R}^6\rt)^*\), the map:
\e\label{U1-htpy}
\bcd[row sep = 0pt]
\ga_\rh: \Un(1) \ar[r]& \ww[-]{3}\lt(\bb{R}^6\rt)^*\\
e^{i\th} \ar[r, maps to]& \cos(\th) \rh + \sin(\th) J_\rh^*\rh
\ecd
\ee
generates the first fundamental group of \(\ww[-]{3}\lt(\bb{R}^6\rt)^*\).  Next, recall that the cohomology group \(\sch{1}{\N;\bb{Z}}\) can be identified with the space of homotopy classes of maps \(\N \to \Un(1)\).  Thus, let \(\rh\) be an \slc\ 3-form on \(\N\) representing the homotopy class \([\rh] \in \mc{SL}_\bb{C}(\N)\), pick some \(\ch' \in r_2^{-1}(\ch) \cc \sch{1}{\N;\bb{Z}}\) and choose some \(f: \N \to \Un(1)\) representing the class \(\ch'\).  Then, \(\ch \1 [\rh] \in \mc{SL}_\bb{C}(\N)\) can be explicitly represented by the \slc\ 3-form \(\rh' = \fr{Re}(f)\rh + \fr{Im}(f) J^*_\rh\rh\).

I now return to the full statement of \tref{SLC-htpy}.  Fix a Riemannian metric \(g\) on \(\N\) and write \(\mc{P}\) for the \(\SO(6)\)-structure on \(\N\) induced by \(g\).  Recall that a spin structure on \(\N\) is a principal \(\Spin(6)\)-bundle \(\mc{Q}\) together with a 2-sheeted covering map \(q: \mc{Q} \to \mc{P}\) such that the following diagram commutes:
\ew
\bcd
& \hs{3mm} \mc{Q} \ar[loop right] \hs{3mm} \ar[dl] \ar[dd, "q"] & \Spin(6) \ar[dd]\\
\N & &\\
& \hs{3mm} \mc{P} \ar[loop right] \hs{3mm} \ar[ul] & \SO(6)
\ecd
\eew
(The reader may wish to note that the bundle \(\mc{Q}\) alone does not determine the map \(q\); see, e.g.\ \cite[p.\ 84, Remark 1.14]{SG}.)  I now prove \tref{SLC-htpy}.

\begin{proof}[Proof of \tref{SLC-htpy}]
Firstly, note that homotopy classes of \slc\ 3-forms (equivalently, \slc-structures) on \(\N\) correspond bijectively to homotopy classes of principal \(\SU(3)\)-subbundles of \(\mc{P}\).  Indeed, writing \(\mc{F}_+\N\) for the oriented frame bundle of \(\N\), \slc-structures on \(\N\) are equivalent to sections of the bundle \(\rqt{\mc{F}_+\N}{\SL(3;\bb{C})}\), and likewise principal \(\SU(3)\)-subbundles \(\mc{P}\) are equivalent to sections of the bundle \(\rqt{\mc{P}}{\SU(3)}\).  The equivalence now follows from the observation that the fibres of the natural map \(\rqt{\mc{F}_+\N}{\SL(3;\bb{C})} \to \rqt{\mc{P}}{\SU(3)}\) are contractible.  Thus, to complete the proof of \tref{SLC-htpy}, it suffices to prove that there exists a map \(\si\) from homotopy classes of \(\SU(3)\)-subbundles of \(\mc{P}\) to spin structures on \((\N,g)\) and that \(\si\) is bijective.

The existence of the map \(\si\) is essentially well-known (see, e.g. \cite[Prop.\ 3.6.2]{CMwSH} for a related result).  Indeed, let \(\mc{R} \pc \mc{P}\) be an \(\SU(3)\)-subbundle.  Consider the diagram:
\e\label{lifting-diagram}
\bcd
& \Spin(6) \ar[d, "\pi"]\\
\SU(3) \ar[ur, dashed, hook, "\vrh"] \ar[r, hook, "\io"] & \SO(6)
\ecd
\ee
Since \(\SU(3)\) is simply connected, Covering Space Theory implies that there is a unique homomorphism \(\SU(3) \oto{\vrh} \Spin(6)\) lifting the inclusion \(\SU(3) \overset{\io}{\emb} \SO(6)\) along the homomorphism \(\pi: \Spin(6) \to \SO(6)\).  Diagram \eqref{lifting-diagram} induces a diagram of bundles:
\ew
\bcd
& \mc{R} \x_\vrh \Spin(6) \ar[d, "q"]\\
\mc{R} \ar[ur, hook] \ar[r, hook] & \mc{R} \x_\io \SO(6) \cong \mc{P}
\ecd
\eew
and thus, setting \(\mc{Q} = \mc{R} \x_\vrh \Spin(6)\) (together with the natural map \(q\) induced by \(\pi:\Spin(6) \to \SO(6)\)), it has been shown that every \(\SU(3)\)-subbundle \(\mc{R} \pc \mc{P}\) canonically induces a spin structure on \(\N\), which clearly depends only on the homotopy class of \(\mc{R}\), thus defining the map \(\si\).

Before proving that \(\si\) is bijective, it is useful to note that the spin structure induced by \(\mc{R}\) may alternatively be characterised as follows.  Given any choice of spin structure \((\mc{Q},q)\) on \((\N,g)\), the bundle \(\mc{R}' = q^{-1}(\mc{R})\) defines an \(\lt(\SU(3) \x \{\pm1\}\rt)\)-subbundle of \(\mc{Q}\).  Clearly if \((\mc{Q},q)\) is the spin structure induced by \(\mc{R}\), then \(q:\mc{R}' \to \mc{R}\) is a disconnected degree 2 cover, i.e.\ \(\mc{R}' \cong \mc{R} \x \{\pm1\}\).  Conversely, if \(\mc{R}' \cong \mc{R} \x \{\pm1\}\), then:
\ew
\mc{Q} \cong \mc{R}' \x_{\lt(\SU(3) \x \{\pm1\}\rt)} \Spin(6) \cong (\mc{R} \x \{\pm1\}) \x_{\lt(\SU(3) \x \{\pm1\}\rt)} \Spin(6) \cong \mc{R} \x_{\SU(3)} \Spin(6)
\eew
with \(q\) defined accordingly, and thus \((\mc{Q},q)\) is precisely the spin structure on \(\N\) induced by \(\mc{R}\).

Using this observation, I now prove that \(\si\) is bijective.  Given a choice of spin structure \((\mc{Q},q)\) on \(\N\), consider the bundle \(\rqt{\mc{Q}}{\SU(3)}\), where one identifies \(\SU(3) \pc \Spin(6)\) via \(\vrh\), and observe that sections of the bundle \(\rqt{\mc{Q}}{\SU(3)}\) correspond to \(\SU(3)\)-subbundles of \(\mc{Q}\).  Since \(\Spin(6) \cong \SU(4)\) (see \cite[Ch.\ I, Thm.\ 8.1]{SG}), \(\rqt{\Spin(6)}{\SU(3)} \cong \rqt{\SU(4)}{\SU(3)} \cong S^7\) and thus it follows from classical Obstruction Theory (see \cite[Thms.\ 6.11 \& 6.12]{EoHT}) that \(\mc{Q}\) admits an \(\SU(3)\)-subbundle and that any two such subbundles are homotopic.  Given such a subbundle \(\mc{R}'\), the image \(\mc{R} = q(\mc{R}') \pc \mc{P}\) defines an \(\SU(3)\)-subbundle of \(\mc{P}\) and since \(q^{-1}(\mc{R}) = \mc{R}' \coprod -\mc{R}' \cong \mc{R} \x \{\pm1\}\), \((\mc{Q},q)\) is precisely the spin structure induced by \(\mc{R}\); thus the map \(\si\) is surjective.  Moreover, since homotopic \(\SU(3)\)-subbundles of \(\mc{Q}\) give rise to homotopic \(\SU(3)\)-subbundles of \(\mc{P}\), the injectivity of \(\si\) is now clear.

\end{proof}

\begin{Rks}~

1. The above argument provides an alternative proof that \(\sch{1}{\N;\rqt{\bb{Z}}{2\bb{Z}}}\) acts faithfully and transitively on \(\mc{SL}_\bb{C}(\N)\), by using the well known result (see \cite[p.\ 82, Thm.\ 1.7]{SG}) that the set of spin structures on a spin manifold \(\N\) form a torsor for the group \(\sch{1}{\N;\rqt{\bb{Z}}{2\bb{Z}}}\).

2. Returning to the perspective of classical Obstruction Theory, recall from \cite[Thm.\ 6.11]{EoHT} that the primary (and, in this case, only) obstruction to the existence of a section of \(\ww[-]{3}\T^*\N\) is determined by an obstruction class:
\ew
\ga \in \sch{2}{\N;\pi_1\lt(\ww[-]{3}\T^*_\pt\N\rt)} \cong \sch{2}{\N;\rqt{\bb{Z}}{2\bb{Z}}}.
\eew
\tref{SLC-htpy} shows that \(\ga\) is simply the second Stiefel--Whitney class of \(\N\).\\
\end{Rks}

Note that the bundle \(\rqt{\mc{Q}}{\SU(3)}\) arising in the above proof is essentially the unit sphere bundle in the spinor bundle \(\mc{S}(\N) = \mc{Q} \x_{\Spin(6)} \bb{C}^4\) associated to \((\mc{Q},q)\), where \(\Spin(6)\) acts on \(\bb{C}^4\) via the identification \(\Spin(6) \cong \SU(4)\).  Using this observation, it is possible to provide a very explicit description of the correspondence between homotopy classes of \slc\ 3-forms and spin structures.  Indeed, fix a choice of spin structure \((\mc{Q},q)\) on \(\N\) and observe that the rank 10 complex vector bundle \(\ss[\bb{C}]{2}\mc{S}(\N)\) is isomorphic to the bundle \(\ww[\bb{C}SD]{3}\T^*\N\) of complex self-dual 3-forms, i.e.\ 3-forms \(\al\) satisfying \(\Hs\al = i\al\), where \(\Hs\) denotes the Hodge star induced by the metric \(g\).\footnote{The author wishes to thank Nigel Hitchin for bringing this isomorphism to his attention.}  Given a non-zero section \(\vsi\) of \(\mc{S}\), the section \(\vsi \ts \vsi \in \ss[\bb{C}]{2}\mc{S}(\N)\) corresponds to a complex 3-form \(\al_\vsi\) and thus to a real 3-form \(\rh_\vsi = \al_\vsi + \-{\al_\vsi}\).  Since the stabiliser of \(\vsi\) in \(\Spin(6)\) at each point of \(\N\) is isomorphic to \(\SU(3)\), the stabiliser of \(\rh_\vsi\) in \(\SO(6)\) at each point of \(\N\) is also isomorphic to \(\SU(3)\), and thus \(\rh_\vsi\) is an \slc\ 3-form such that the metric \(g\) is Hermitian \wrt\ \(J_\rh\).  Since all non-zero sections of \(\mc{S}(\N)\) are homotopic, it is immediately clear that all \slc\ 3-forms obtained in this way are likewise homotopic.

Conversely, given an \slc\ 3-form \(\rh\), choose a Hermitian metric \(g\) on \(\N\) (\wrt\ \(J_\rh\)).  For each spin structure \((\mc{Q},q)\) on \((\N,g)\), there is a unique section \(\vsi\) of the bundle \(\rqt{\mc{S}(\N)\osr\N}{\{\pm1\}}\) such that \(\rh = \rh_\vsi\) (which is well-defined, since \(\rh_{-s} = \rh_s\) for any non-zero spinor \(s\)).  It follows from the proof of \tref{SLC-htpy} that there is a unique spin structure \((\mc{Q},q)\) such that the section \(\vsi\) of \(\rqt{\mc{S}(\N)\osr\N}{\{\pm1\}}\) can be lifted to a section of \(\mc{S}(\N)\osr\N\) and this is precisely the spin structure induced by \(\rh\).

I end this subsection by providing some explicit examples.\\

\begin{Exs}~

1. Consider the torus \(\N = \bb{T}^6\) and let \(\rh_-\) denote the `standard' \slc\ 3-form on \(\bb{T}^6\), defined by identifying \(\T\lt(\bb{T}^6\rt)\) with \(\bb{T}^6 \x \bb{R}^6\).  Since \(\sch{1}{\bb{T}^6; \rqt{\bb{Z}}{2\bb{Z}}} \cong \lt(\rqt{\bb{Z}}{2\bb{Z}}\rt)^6\), \(\bb{T}^6\) admits \(2^6 = 64\) distinct homotopy classes of \slc\ 3-forms.  Moreover, since the map \(\sch{1}{\bb{T}^6; \bb{Z}} \to \sch{1}{\bb{T}^6; \rqt{\bb{Z}}{2\bb{Z}}}\) is surjective, one may provide an explicit description of all 64 classes as follows.  Let \(\lt(x^1,...,x^6\rt)\) denote the canonical periodic coordinates on \(\bb{T}^6 \cong \rqt{\bb{R}^6}{\bb{Z}^6}\) and, for each \(a = (a_1,...,a_6) \in \lt(\rqt{\bb{Z}}{2\bb{Z}}\rt)^6 \cong \sch{1}{\bb{T}^6; \rqt{\bb{Z}}{2\bb{Z}}}\), consider the map:
\ew
\bcd[row sep = 0pt]
f_a: \bb{T}^6 \ar[r]& \Un(1)\\
\lt(x^1,...,x^6\rt) \ar[r, maps to]& \exp\lt(i\sum_{j=1}^6 a_j x^j\rt).
\ecd
\eew
The map \(f_a\) represents a cohomology class in \(\sch{1}{\bb{T}^6;\bb{Z}}\) which maps to \(a\) under \(r_2: \sch{1}{\bb{T}^6;\bb{Z}} \to \sch{1}{\bb{T}^6; \rqt{\bb{Z}}{2\bb{Z}}}\).  It follows that the 64 homotopy classes of \slc\ 3-forms on \(\bb{T}^6\) can be explicitly represented by the 3-forms:
\ew
\rh_a = \cos\lt(\sum_{j=1}^6 a_j x^j\rt)\rh_- + \sin\lt(\sum_{j=1}^6 a_j x^j\rt)J_{\rh_-}\rh_-.
\eew

2. Let \((\N,J,g)\) be a compact, Hermitian manifold.  By \cite[Prop.\ 3.2]{RS&SS} (see also \cite[Thm.\ 2.2]{HS}), there is a bijective correspondence between spin structures on \(\N\) and holomorphic square roots of the canonical bundle \(\ww{3,0}\T^*\N\).  Thus, given an \slc\ 3-form \(\rh\) on \(\N\) such that:
\e\label{compatible}
J_\rh = J,
\ee
Theorem 9.4.1 predicts that \(\rh\) defines a holomorphic square root of \(\ww{3,0}\T^*\N\).

In the case where \(\ww{3,0}\T^*\N \cong \mc{O}\) (i.e.\ \((\N,J,g)\) has trivial canonical bundle) this may be seen directly as follows.  Initially, let \(\rh\) be an \slc\ 3-form on \(\N\) satisfying \eref{compatible} such that \(\dd\rh = \dd J^*\rh = 0\).  Then, \(\Om = \rh + iJ^*\rh\) defines a non-zero holomorphic \((3,0)\)-form on \(\N\), hence a holomorphic trivialisation of \(\ww{3,0}\T^*\N\) and whence a natural square root of \(\ww{3,0}\T^*\N\), {\it viz.}\ \(\mc{O}\).

Now, let \(\rh'\) be an arbitrary \slc\ 3-form satisfying \eref{compatible}.  Firstly, note that \(\rh'\) canonically defines a class \(\de_{\rh'}\) in \(\sch{1}{\N;\rqt{\bb{Z}}{2\bb{Z}}}\).  Indeed, \(\rh'\) defines a unique map \(f_{\rh'}: \N \to \bb{C}\osr\{0\}\) via \(\rh' + iJ^*\rh' = f\Om\).  Define \(\de_{\rh'}\) to be the reduction modulo 2 of the pullback of the canonical generator \({\bf 1} \in \sch{1}{\bb{C}\osr\{0\};\bb{Z}}\) along the map \(f_{\rh'}\).  Next, by \cite[p.\ 15]{HS}, the space of holomophic square roots of \(\ww{3,0}\T^*\N\) is naturally a torsor for the group \(\sch{1}{\N; \rqt{\bb{Z}}{2\bb{Z}}}\).  Indeed, the short exact sequence of sheaves:
\ew
\bcd[row sep = 0pt]
{\bf 1} \ar[r] & \rqt{\bb{Z}}{2\bb{Z}} \ar[r] & \mc{O}^* \ar[r] & \mc{O}^* \ar[r] & {\bf 1}\\
& & f \ar[r, maps to] & f^2 &
\ecd
\eew
induces a sequence:
\ew
\bcd[row sep = 0pt]
0 \ar[r] & \sch{1}{\N; \rqt{\bb{Z}}{2\bb{Z}}} \ar[r] & \sch{1}{\N;\mc{O}^*} \ar[r] & \sch{1}{\N;\mc{O}^*}\\
& & L \ar[r, maps to] & L^{\ts2}
\ecd
\eew
as claimed (recalling that \(\sch{1}{\N;\mc{O}^*}\) may be identified with the Picard group of \(\N\)).  The square root of \(\ww{3,0}\T^*\N\) defined by the \slc\ 3-form \(\rh'\) is then simply \(\de_{\rh'}\1\mc{O}\).

I end this subsection by remarking on one interesting aspect of this case.  By Theorem 9.4.1, two \slc\ 3-forms \(\rh'\) and \(\rh''\) satisfying \eref{compatible} are homotopic through arbitrary \slc\ 3-forms \iff\ \(\de_{\rh'} = \de_{\rh''}\).  However, clearly \(\rh'\) and \(\rh''\) are homotopic through \slc\ 3-forms satisfying \eref{compatible} \iff\ the induced maps \(f_{\rh'}, f_{\rh''}:\N \to \bb{C}\osr\{0\}\) are homotopic, which occurs \iff\ the classes \(f_{\rh'}^*{\bf 1}\) and \(f_{\rh''}^*{\bf 1}\) in \(\sch{1}{\N;\bb{Z}}\) coincide.  Thus, in general, there exist pairs of homotopic \slc\ 3-forms, each satisfying \eref{compatible}, which nevertheless cannot be connected by any path of \slc\ 3-forms satisfying \eref{compatible}.\\
\end{Exs}

\subsection{\boldmath Extendibility of \(\SL(3;\bb{C})\) 3-forms}\label{SLC-Ext}

By \dref{ext-def}, \(\rh\) is extendible \iff\ the almost complex manifold \((\N,J_\rh)\) admits a pseudo-Hermitian metric of (real) signature \((2,4)\).  In general, the existence of metrics of indefinite signature is an open problem and thus completely classifying when \(\rh\) is extendible appears unfeasible at this time.  Nevertheless, much insight into the extendibility of \(\rh\) can be gained from the following proposition:

\begin{Prop}\label{C-ext-prop}
Let \(\N\) be an oriented 6-manifold and \(\rh\) an \slc\ 3-form on \(\N\).  Then, \(\rh\) is extendible \iff\ the (complex) projectivised tangent bundle of \(\N\), \(\bb{P}_\bb{C}(\T\N,J_\rh)\), admits a global section, i.e.\ \((\T\N,J_\rh)\) admits a complex line subbundle.
\end{Prop}

\begin{proof}
Initially, suppose that \(\cal{L} \pc (\T\N,J_\rh)\) is a complex line subbundle.  Let \(g\) be any Hermitian metric (of real signature \((3,0)\)) on \(\N\) and write:
\ew
\T\N = \cal{L} \ds \cal{L}^\bot, \hs{3mm} g = g|_\cal{L} + g|_{\cal{L}^\bot},
\eew
where the orthocomplement is defined \wrt\ \(g\).  Then, \(g_\cal{L} - g_\bot\) is a pseudo-Hermitian metric of (real) signature \((2,4)\).

Conversely, let \(\tld{g}\) be a pseudo-Hermitian metric of real signature \((2,4)\).  Define a subbundle \(\Pi_{\tld{g}} \pc \bb{P}_\bb{C}(\T\N,J_\rh)\) via:
\ew
\lt.\Pi_{\tld{g}}\rt|_p = \lt\{\cal{L}\in\lt.\bb{P}_\bb{C}(\T\N,J_\rh)\rt|_p ~\m|~ \tld{g} \text{ is positive definite on }\cal{L}\pc\T_p\N\rt\}.
\eew
Given \(\cal{L} \in \lt.\Pi_{\tld{g}}\rt|_p\), every other \(\mc{L}' \in \lt.\Pi_{\tld{g}}\rt|_p\) can be written as a graph over \(\cal{L}\).  Thus, \(\Pi_{\tld{g}}\) has contractible fibres and hence admits a global section, and whence so does \(\bb{P}_\bb{C}(\T\N,J_\rh)\).

\end{proof}

I now prove \tref{slc-ext-conds}.  Recall the statement of the theorem:\vs{3mm}

\noindent{\bf Theorem \ref{slc-ext-conds}.}
\em Let \(\N\) be an oriented 6-manifold.  If the Euler class \(e(\N) = 0\), then any \slc\ 3-form on \(\N\) is extendible.  In particular:
\begin{itemize}
\item If \(\N\) is open, then any \slc\ 3-form on \(\N\) is extendible.
\item If \(\N\) is closed and the Euler characteristic \(\ch(\N) = 0\), then any \slc\ 3-form on \(\N\) is extendible.
\end{itemize}
Conversely, if \(e(\N) \ne 0\) and in addition \(b^2(\N) = 0\) (i.e.\ \(\sch{2}{\N;\bb{Z}}\) and \(\sch{4}{\N;\bb{Z}}\) are pure torsion), then no \slc\ 3-form on \(\N\) is extendible.\vs{2mm}\em

\begin{proof}
Firstly note that if \(e(\N) = 0\), then \(\N\) admits a nowhere vanishing vector field.  Indeed, if \(\N\) is closed this follows from \cite{VinDM}, whereas if \(\N\) is open, this follows since \(\N\) deformation retracts onto a subcomplex of its \(5\)-skeleton and every rank 6 vector bundle over a 5-dimensional cell complex admits a nowhere vanishing section.  Thus, let \(X\) be a nowhere vanishing vector field on \(\N\) and let \(\rh\) be an \slc\ 3-form on \(\N\).  Then, the real rank 2 distribution on \(\N\) generated by \(X\) and \(J_\rh X\) defines a complex line subbundle of \((\T\N,J_\rh)\) and hence \(\rh\) is extendible, by \pref{C-ext-prop}.

Conversely, suppose \(e(\N) \ne 0\) (in particular, \(\N\) must be closed) and let \(\rh\) be an extendible \slc\ 3-form on \(\N\).  By \pref{C-ext-prop}, one can write \((\T\N,J_\rh) = \cal{L}_1 \ds \cal{L}_2\) with \(\cal{L}_{1,2}\) complex subbundles of \((\T\N,J_\rh)\) of (complex) ranks 1 and 2 respectively.  Then:
\ew
e(\N) = e(\cal{L}_1) \cup e(\cal{L}_2) \in \sch{6}{\N;\bb{Z}} \cong \bb{Z},
\eew
where \(\cup\) denotes the usual cup-product on cohomology.  Since \(e(\N)\) is non-zero, neither \(e(\cal{L}_1)\) nor \(e(\cal{L}_2)\) can have finite order (since \(\bb{Z}\) is torsion-free).  Thus, \(b^2(\N) \ne 0\), as claimed.

\end{proof}

Using \tref{slc-ext-conds}, it is possible to give many examples of extendible and non-extendible \slc\ 3-forms.\\

\begin{Exs}~\vs{-1mm}
\begin{enumerate}
\item Let \(\K\) be any closed, oriented (connected)  spin 5-manifold and set \(\N = S^1 \x \K\). Then, \(\N\) is also oriented and spin.  Thus, since \(\ch(\N) = 0\) and:
\ew
\sch{1}{S^1 \x \K;\rqt{\bb{Z}}{2\bb{Z}}} \cong \sch{1}{\K;\rqt{\bb{Z}}{2\bb{Z}}} \ds \sch{0}{\K;\rqt{\bb{Z}}{2\bb{Z}}} \cong \sch{1}{\K;\rqt{\bb{Z}}{2\bb{Z}}} \ds \rqt{\bb{Z}}{2\bb{Z}},
\eew
by \tref{SLC-htpy}, the manifold \(\N\) admits \(2^{1+b^1\lt(\N; \lt.\raisebox{1mm}{\(\mfn{\bb{Z}}\)}\middle/\raisebox{-1mm}{\(\mfn{2}\)}\rt.\rt)} \ge 2\) distinct homotopy classes of \slc\ 3-forms, all of which are extendible by \tref{slc-ext-conds}.   As a special case, \(\bb{T}^6\) admits \(2^6 = 64\) distinct homotopy classes of extendible \slc\ 3-forms.

\item Consider the sphere \(S^6\). Clearly \(S^6\) is orientable and spin, and \(\sch{1}{S^6;\rqt{\bb{Z}}{2\bb{Z}}} = 0\).  Thus, \(S^6\) admits a unique homotopy class of \slc\ 3-forms, which is not extendible since \(\ch(S^6) = 2\) and \(b^2(S^6) = 0\).  In particular, consider the `standard' \slc\ 3-form on \(S^6\) induced by the embedding \(S^6\emb\bb{R}^7\), where \(\bb{R}^7\) is equipped with its standard (flat) \g\ 3-form \(\ph_0\).  Then, \(\ph_0|_{S^6}\) is not extendible (to a \sg\ 3-form).

\item Consider the manifold \(Y_1 = \bb{RP}^3 \x \bb{RP}^3\) and let \(Y_n = \underbrace{Y_1 \# ... \# Y_1}_{\text{\(n\) times}}\), where \(\#\) denotes connected sum.  Then, \(Y_n\) is spin and \(\sch{1}{Y_n; \rqt{\bb{Z}}{2\bb{Z}}} \cong \lt(\rqt{\bb{Z}}{2\bb{Z}}\rt)^{2n}\), so \(Y_n\) admits \(2^{2n}\) distinct homotopy classes of \slc\ 3-forms.  However, the Betti numbers of \(Y_n\) are:
\ew
\lt(b^0, b^1,b^2,b^3,b^4,b^5,b^6\rt) = (1,0,0,2n,0,0,1)
\eew
and so \(\ch(Y_n) = 2-2n\), whilst \(b^2(Y_n) = 0\).  Thus, for \(n > 1\), none of the \(2^{2n}\) homotopy classes of \slc\ 3-forms on \(Y_n\) are extendible.\\
\end{enumerate}
\end{Exs}

\section{Topological properties of \slr\ 3-forms}\label{slr}

\subsection{\boldmath Homotopic \(\SL(3;\bb{R})^2\) 3-forms}\label{htpic-slr}

Let \(\N\) be an oriented 6-manifold.
\begin{Prop}\label{slr-oGr3}
Write \(\mc{SL}_\bb{R}(\N)\) for the set of homotopy classes of \slr\ 3-forms on \(\N\) and \(\tld{\mc{G}{\kern-0.7pt}r}_3(\N)\) for the set of homotopy classes of sections of \(\oGr_3(\N)\). Then, there is a bijective correspondence:
\ew
\mc{E}:\mc{SL}_\bb{R}(\N) \to \tld{\mc{G}{\kern-0.7pt}r}_3(\N)
\eew
given by \([\rh] \mt [E_{+,\rh}]\).  In particular, \(\N\) admits \slr\ 3-forms \iff\ it admits oriented rank 3 distributions.  The same conclusions apply to homotopy classes of closed \slr\ 3-forms, or to homotopy classes of closed \slr\ 3-forms representing a fixed cohomology class.
\end{Prop}

\begin{proof}
Write \(\ww[+]{3}\T^*\N\) for the bundle of \slr\ 3-forms over \(\N\) and consider the diagram:
\ew
\bcd
\ww[+]{3}\T^*\N \ar[d, " \rh \mt E_{+,\rh}" '] \ar[dr] &\\
\oGr_3(\T\N) \ar[r] & \N
\ecd
\eew
Then, the left-hand map is a fibration, with fibre diffeomorphic to \(\rqt{\Stab_{\GL_+(6;\bb{R})}(E_+)}{\SL(3;\bb{R})^2}\), where \(E_+ = \<e_1,e_2,e_3\?\) denotes the \(+1\)-eigenspace of the standard para-complex structure \(I_0\) on \(\bb{R}^6\) (see \eref{I0}).  Explicitly:
\ew
\Stab_{\GL_+(6;\bb{R})}(E_+) = \lt\{ \bpm A & C\\  & B \epm ~\m|~ A,B\in\GL_+(3;\bb{R}),  C\in\fr{gl}(3;\bb{R})\rt\}.
\eew
The result thus follows, since the quotient \(\rqt{\Stab_{\GL_+(6;\bb{R})}(E_+)}{\SL(3;\bb{R})^2}\) is contractible.  The final remark follows from the homotopy equivalence in \eref{HE}.

\end{proof}

Explicitly, the inverse to \(\mc{E}\) can be described as follows: given an oriented rank 3 distribution \(E\) on \(\N\), choose a distribution \(E'\) such that \(\ts \N = E \ds E'\).  Since \(E\) and \(\ts \N\) are oriented, so is \(E'\) and thus one can pick volume elements \(\vpi_\pm\in\ww{3}E^*_\pm\). Using the inclusion:
\ew
\ww{3}E^*_+ \ds \ww{3}E^*_- \emb \ww{3}(E_+ \ds E_-)^* \cong \ww{3}\T^*\lt(\bb{T}^6\rt),
\eew
one may regard \(\rh = \vpi_+ + \vpi_-\) as a 3-form on \(\bb{T}^6\). It is then simple to verify that \(\rh\) is an \slr\ 3-form on \(\bb{T}^6\) such that \(E_\pm\) are the \(\pm1\)-eigenbundles of the para-complex structure \(I_\rh\).

I now prove \tref{slr-htpy}.  Recall the statement of the theorem:\vs{3mm}

\noindent{\bf Theorem \ref{slr-htpy}.}
\em Let \(\N\) be a closed, oriented, 6-manifold with \(e(\N) = 0\) and suppose \(w_2(\N)^2 = 0\).  Write \(\rh_2:\sch{4}{\N;\bb{Z}} \to \sch{4}{\N;\rqt{\bb{Z}}{2\bb{Z}}}\) for reduction modulo 2 and define:
\ew
\sch{4}{\N;\bb{Z}}_{\bot w_2} = \lt\{ u \in \sch{4}{\N;\bb{Z}} ~\m|~ \rh_2u \cup w_2(\N) = 0 \rt\}.
\eew
Then, there is an injection from \(\rqt{\sch{4}{\N;\bb{Z}}_{\bot w_2}}{\text{2-torsion}}\) into the set of homotopy classes of \slr\ 3-forms on \(\N\) (equivalently closed \slr\ 3-forms, or \slr\ 3-forms in any fixed degree 3 cohomology class).  In particular, if \(\N\) is spin and \(b^4(\N)>0\), then each of these sets is infinite.\vs{2mm}\em

I remark that in order for \(\N\) to admit an oriented rank 3 distribution, the Euler class \(e(\N)\) must clearly vanish.  Thus, the significant hypothesis in \tref{slr-htpy} is \(w_2(\N)^2 = 0\).

\begin{proof}[Proof of \tref{slr-htpy}]
Recall the first spin characteristic class \(q_1\) defined by Thomas in \cite[Thm.\ 1.2]{OtCGotCSftSSG}, which is related to the first Pontryagin class \(p_1\) by \(p_1 = 2q_1\).  Since \(e(\N) = 0\) and \(w_2(\N)^2 = 0\), \cite[Cor.\ 1.7]{FoTkPoM} implies that for every \(u \in \sch{4}{\N;\bb{Z}}_{\bot w_2}\), there exists an oriented, spin, rank 3 distribution \(E\) on \(\N\) with \(q_1(E) = 2u\).  Moreover, given classes \(u, u' \in \sch{4}{\N;\bb{Z}}_{\bot w_2}\) with corresponding bundles \(E\) and \(E'\), note that if \(E\) and \(E'\) are homotopic as sections of \(\oGr_3(\T\N)\), then \(q_1(E) = q_1(E')\), and hence \(2(u-u') = 0\).  The result follows.

\end{proof}

\noindent \cref{spin-cor} now follows at once, by restricting attention to the case \(w_2(\N) = 0\).  I remark also that \(w_2(\N) = 0\) is necessary for \(\N\) to admit any extendible \slr\ 3-forms and is thus a natural condition to impose when studying \slr\ 3-forms.  (To see this, simply combine \pref{ext-off} and \tref{SG2-Exist}, together with the fact that the boundary of a spin manifold is also spin.)

Using the above results, one can give many examples of manifolds admitting multiple homotopy classes of \slr\ 3-forms.

\begin{Exs}\label{SLR-ext-exs}~\vs{-1mm}

\begin{enumerate}
\item As a simple example, take \(\N = \bb{T}^6\).  \(\bb{T}^6\) is parallelisable and so trivially it is orientable, spin and has vanishing Euler class.  Thus, \(\sch{4}{\bb{T}^6;\bb{Z}}_{\bot w_2} = \sch{4}{\bb{T}^6;\bb{Z}} \cong \bb{Z}^{15}\) and hence \(\bb{T}^6\) admits infinitely many distinct homotopy classes of \slr\ 3-forms.

\item Now, consider \(\N = S^6\). Since \(\ch(S^6) = 2\), \(S^6\) admits no \slr\ 3-forms.  Likewise, the manifolds \(Y_n\) (\(n > 1\)) considered in \sref{SLC-Ext} admit no \slr\ 3-forms.

\item Let \(\N = S^1 \x \K\), where \(\K\) is any closed, orientable, spin 5-manifold. Then, \(S^1\x \K\) is also spin and has vanishing Euler class. Thus, \(S^1\x \K\) admits \slr\ 3-forms. Moreover, if \(b^4(\N)=0\), then the relation \(b^n(\N) = b^n(\K) + b^{n-1}(\K)\) forces \(b^3(\K) = b^4(\K) = 0\), and hence \(b^1(\K) = b^2(\K) = 0\), too, by Poincar\'{e} duality, i.e. \(\K\) is a rational homology sphere.  Thus, unless \(\K\) is a rational homology sphere, \(\N\) admits infinitely many distinct homotopy classes of \slr\ 3-forms.

\item Let \(\mc{E}\) denote the Enriques surface -- i.e.\ the quotient of a K3 surface by a fixed-point-free holomorphic involution -- viewed as a real 4-manifold.  Then, the 6-manifold \(\N = \bb{T}^2 \x \mc{E}\) provides an example of a non-spin manifold which admits \slr\ 3-forms; in fact, \(\N\) admits an infinite number of distinct homotopy classes of \slr\ 3-forms, none of which can be extendible, as remarked above.

To verify this, it is necessary to recall some results on the topology of \(\mc{E}\).  Recall that \cite[Lem.\ 15.1]{CCS}, \cite[\S6.10]{RES}\footnote{Note a slight error in this second reference: \(h^{2,0}(\mc{E}) = 0\), not 1 as stated {\it loc.\ cit.}.}:
\ew
\sch{2}{\mc{E};\bb{Z}} \cong \bb{Z}^{10} \ds \rqt{\bb{Z}}{2\bb{Z}}
\eew
where the \(\rqt{\bb{Z}}{2\bb{Z}}\)-factor is generated by \(c_1(\mc{E})\) and recall also that the Betti numbers of \(\mc{E}\) are \(b^0(\mc{E}) = b^4(\mc{E}) = 1\), \(b^1(\mc{E}) = b^3(\mc{E}) = 0\) and \(b^2(\mc{E}) = 10\).  The restriction of \(c_1(\mc{E})\) modulo 2 is non-zero, hence \(w_2(\mc{E}) \ne 0\) and whence \(\mc{E}\), and also \(\N\), are not spin.  Nevertheless \(\ch(\mc{E}) = 12\) which vanishes modulo 2 and thus \(w_2(\mc{E})^{\cup2} = w_4(\mc{E}) = 0\).  Hence, by \tref{slr-htpy}, \(\N = \bb{T}^2 \x \mc{E}\) admits \slr\ 3-forms.  Moreover, given a class \(u \in \sch{4}{\N;\bb{Z}}\), identifying \(c_1(\mc{E})\) with an integral degree 2 cohomology class on \(\N\) in the natural way, one finds that:
\ew
2(u \cup c_1(\mc{E})) = u \cup 2c_1(\mc{E}) = 0
\eew
and hence \(u \cup c_1(\mc{E}) = 0\), since \(\sch{6}{\N;\bb{Z}}\) has no 2-torsion.  Since \(w_2(\N) = \rh_2c_1(\mc{E})\) it follows that:
\ew
\rh_2u \cup w_2(\N) = \rh_2(u \cup c_1(\mc{E})) = 0,
\eew
and thus \(\sch{4}{\N;\bb{Z}}_{\bot w_2} = \sch{4}{\N;\bb{Z}}\).  Since \(b^4(\N) = 11\), it follows that \(\mc{SL}_\bb{R}(\N)\) is infinite, as claimed.\\
\end{enumerate}
\end{Exs}

\subsection{\boldmath Examples of extendible \(\SL(3;\bb{R})^2\) 3-forms}\label{ext-slr}

The previous subsection saw examples of non-extendible \slr\ 3-forms.  This section provides explicit examples of extendible \slr\ 3-forms.  I begin with a preparatory result:

\begin{Prop}\label{E-iso}
Let \(\N\) be an oriented 6-manifold and let \(\rh\) an \slr\ 3-form on \(\N\). Then, \(\rh\) is extendible \iff\ the bundle \(\Iso(E_{+,\rh},E_{-,\rh})\) of isomorphisms from \(E_{+,\rh}\) to \(E_{-,\rh}\) has a global section.
\end{Prop}

\begin{proof}
By \dref{ext-def}, \(\rh\) is extendible \iff\ \(\N\) admits a 2-form \(\om\) satisfying \(\om^3 <0\) such that \(\cal{I}_\rh\om\) is a pseudo-Riemannian metric of signature \((3,3)\).  Thus, suppose \(\rh\) is extendible and define \(\tld{g} = \cal{I}_\rh\om\). One may verify that \(\tld{g}(I_\rh\cdot,I_\rh\cdot) = -\tld{g}(\cdot,\cdot)\).  Thus, given \(u\in E_{+,\rh}\), for any other \(w\in E_{+,\rh}\) one has:
\ew
\tld{g}(u,w) = \tld{g}(I_\rh u,I_\rh w) = -\tld{g}(u,w) = 0
\eew
and hence \(\tld{g}(v,\cdot)\) may naturally be identified with an element of \(\lt(E_{-,\rh}\rt)^*\).

Now, choose a positive definite metric \(h\) on \(E_{-,\rh}\) and define \(L(v)\in E_{-,\rh}\) to be the unique element such that:
\ew
h(L(v),\cdot) = \tld{g}(v,\cdot)\in \lt(E_{-,\rh}\rt)^*.
\eew
Then, \(L\) defines an isomorphism \(E_{+,\rh} \to E_{-,\rh}\).

Conversely, let \(L:E_{+,\rh} \to E_{-,\rh}\) be a fibrewise isomorphism. Define subbundles \(F_\pm\pc\T\N\) by:
\ew
F_\pm = (\Id \pm L)(F_{+,\rh}).
\eew
Then, \(\T\N = F_+ \ds F_-\) and \(I_\rh\) maps \(F_+\) isomorphically onto \(F_-\) and {\it vice versa}.

Now, choose any positive definite metric \(l\) on \(F_+\), extend it to a symmetric bilinear form on \(\T\N\) by setting \(l(F_-,\cdot) = 0\) and define:
\ew
\tld{g}(\cdot,\cdot) = l(\cdot,\cdot) - l(I_\rh\cdot,I_\rh\cdot).
\eew
Then, \(\om(\cdot,\cdot) = \tld{g}(I_\rh\cdot,\cdot)\) is a 2-form and \(\cal{I}_\rh\om = \tld{g}\) has signature \((3,3)\).  In particular, it follows that \(\om^3 \ne 0\) and hence, by replacing \(\om\) by \(-\om\) if necessary, one may assume that \(\om^3<0\).  Thus, \(\rh\) is extendible.

\end{proof}

I remark that \pref{E-iso} has the following, curious result.  Recall that if \(\rh\) is an \slc\ 3-form, then \(\rh\) is always homotopic to \(-\rh\) (see \eref{U1-htpy}; see also \cite[\S4]{RoG2MwB}).  By applying \pref{E-iso}, a partial analogue for \slr\ 3-forms may be obtained:

\begin{Cor}\label{htpy-to-op}
Let \(\N\) be an oriented 6-manifold and let \(\rh\) be an extendible \slr\ 3-form on \(\N\). Then, \(\rh\) is homotopic (through \slr\ 3-forms) to \(-\rh\).
\end{Cor}

In particular, by the homotopy equivalence in \eref{HE}, if \(\rh\) is closed and extendible, then \(\rh\) is homotopic to \(-\rh\) through closed \slr\ 3-forms and, likewise, if \(\rh\) is exact and extendible, then \(\rh\) is homotopic to \(-\rh\) through exact \slr\ 3-forms.

\begin{proof}
By \pref{slr-oGr3}, it is equivalent to prove that the sections of \(\oGr_3(\N)\) induced by \(E_{+,\rh}\) and \(\ol{E}_{+,\rh}\) (the same bundle equipped with the opposite orientation) are homotopic.  Since \(\rh\) is extendible, one can choose \(L \in \Iso(E_{+,\rh},E_{-,\rh})\).  For \(t \in [0,1]\), define:
\ew
\bcd[row sep = 0pt]
\io_t:E_{+,\rh} \ar[r]& \T\N\\
v \ar[r, maps to]& \cos(\pi t)v + \sin(\pi t)L(v).
\ecd
\eew
Then, \(\io_t(E_{+,\rh})\) defines the required homotopy from \(E_{+,\rh}\) to \(\ol{E}_{+,\rh}\).

\end{proof}

The converse of \cref{htpy-to-op} does not hold.  Indeed, consider the manifold \(\N = \bb{T}^2 \x \mc{E}\) of \exref{SLR-ext-exs}(4).  It is known \cite[Lem.\ 15.1]{CCS}, \cite[\S6.10]{RES} that the intersection form of \(\mc{E}\) is indefinite with signature \(\si(\mc{E}) = 8\); in particular, since \(\ch(\mc{E}) = 12\) the equations:
\ew
\si(\mc{E}) \pm \ch(\mc{E}) \equiv 0 \hs{2mm} \text{mod 4}
\eew
hold.  Using \cite[Thm.\ 2(A)]{Fo2P&2KoACSoC4DM}, it follows that there exists an oriented rank 2 distribution \(\mc{E}\), which I shall denote by \(\pi\).  Write \(\del_1,\del_2\) for the standard basis of constant vector fields on \(\bb{T}^2\) and define an oriented rank 3 distribution \(E\) on \(\N = \bb{T}^2 \x \mc{E}\) via:
\ew
E = \<\del_1\? \ds \pi,
\eew
where \(E\) is oriented according to the ordering of the summands.  Then, \(E\) is manifestly homotopic to \(\ol{E}\): an explicit homotopy is given by:
\ew
E_t = \<\cos(\pi t)\del_1 + \sin(\pi t)\del_2\? \ds \pi.
\eew
Note however that, as observed above, none of the \slr\ 3-forms on \(\N\) are extendible.

I now return to the task of constructing explicit examples of extendible forms.  Recall the following definition:
\begin{Defn}[{See \cite[\S2]{VBOT&NS}}]
Let \(\M\) be an arbitrary manifold. A vector bundle \(\bb{E}\to\M\) is termed flat if it admits some connection of curvature 0. Equivalently, write \(\tld{\M}\) for the universal covering space of \(\M\) and view \(\tld{\M}\to\M\) as a principal \(\pi_1(\M)\)-bundle over \(\M\). Then, a vector bundle \(\bb{E}\to\M\) of rank \(k\) is flat \iff\ it can be written as:
\ew
\bb{E} = \rqt{\tld{\M} \x \bb{R}^k}{\pi_1(\M)}
\eew
for some representation \(\pi_1(\M)\to\GL(k;\bb{R})\).
\end{Defn}

This definition can be made very explicit for \(\bb{T}^6\): identify \(\bb{T}^6\cong\rqt{\bb{R}^6}{\bb{Z}^6}\) so that the quotient map \(\bb{R}^6\to\bb{T}^6\) is also the universal cover. Then, for every representation \(\varrho:\bb{Z}^6\to\GL(k;\bb{R})\), one obtains the following commutative diagram:
\ew
\begin{tikzcd}
\bb{R}^6\x\bb{R}^k \ar[d, "proj_1"] \ar[r, "\text{quot}"] & \rqt{\bb{R}^6 \x \bb{R}^k}{(\bb{Z}^6,\varrho)} = \bb{E} \ar[d]\\
\bb{R}^6 \ar[r, "\text{quot}"] & \rqt{\bb{R}^6}{\bb{Z}^6} \cong \bb{T}^6\\
\end{tikzcd}
\eew
where \(\bb{Z}^6\) acts on \(\bb{R}^6\) by translation and \(\bb{R}^k\) via \(\varrho\). Then, \(\bb{E}\to\bb{T}^6\) is the flat bundle corresponding to the representation \(\varrho\). Moreover, note that \(\T\lt(\bb{T}^6\rt)\) is simply the flat bundle corresponding to the trivial representation \(\varrho:\bb{Z}^6\to\GL(6;\bb{R})\).

The following result was proved in \cite[Thm.\ 3.3]{VBOT&NS}:
\begin{Thm}[Auslander--Szczarba]\label{ASThm}
Let \(\bb{E}_1\) and \(\bb{E}_2\) be flat, rank \(k\) vector bundles over \(\bb{T}^6\). Then, \(\bb{E}_1\) and \(\bb{E}_2\) are isomorphic \iff\ their first and second Stiefel--Whitney classes coincide.
\end{Thm}

I shall restrict attention to flat, orientable bundles; these correspond to representations \(\varrho\) whose image lies in \(\GL_+(k,\bb{R})\).  By \tref{ASThm}, two flat, orientable, rank \(k\) vector bundles over \(\bb{T}^6\) are isomorphic \iff\ their second Stiefel--Whitney classes coincide.  Using this result, one can construct a large number of (non-homotopic) extendible \slr\ 3-forms on \(\bb{T}^6\).

\begin{Cstr}\label{flat-bun-cstr}
By \prefs{slr-oGr3} and \ref{E-iso}, it suffices to construct oriented rank 3 distributions \(E_+ \pc \T\lt(\bb{T}^6\rt)\) such that \(\T\lt(\bb{T}^6\rt) = E_+ \ds E_-\) for some \(E_- \cong E_+\).  Let \(\bb{E}\) represent an isomorphism class of flat, orientable, rank 3 vector bundles over \(\bb{T}^6\) and define \(\bb{T} = \bb{E}_+ \ds \bb{E}_-\), where \(\bb{E}_+ = \bb{E}_- = \bb{E}\) are identical and the subscripts only serve to keep track of the summands. Then \(\bb{T}\) is automatically orientable and hence \(w_1(\bb{T}) = 0\). Moreover:
\ew
w_2(\bb{T}) = 2\cdot w_2(\bb{E}) + w_1(\bb{E})^{\cup2} = 0 \in \sch{2}{\bb{T}^6;\rqt{\bb{Z}}{2\bb{Z}}},
\eew
where \(\cup\) denotes the usual cup-product on cohomology; thus by \tref{ASThm}, \(\bb{T}\) is oriented isomorphic to \(\T\lt(\bb{T}^6\rt)\).  Choose such an isomorphism and write \(E_\pm\) for the images of \(\bb{E}_\pm\) respectively under this isomorphism.  Then, the bundle \(\Iso(E_+,E_-)\) has a natural global section, corresponding to \(\Id\in\Iso(\bb{E}_+,\bb{E}_-)\).
\end{Cstr}

Given flat, orientable, rank 3 bundles \(\bb{E}\) and \(\bb{E}'\) representing different isomorphism classes, by \tref{ASThm} \(w_2(\bb{E}) \ne w_2(\bb{E}')\), hence \(w_2(E_+) \ne w_2(E'_+)\) and whence \(\rh\) and \(\rh'\) represent distinct homotopy classes.  Thus, Construction \ref{flat-bun-cstr} provides an injection from the set of isomorphism classes of flat, orientable, rank 3 bundles over \(\bb{T}^6\) into the set of homotopy classes of \slr\ 3-forms.  Moreover, the former set may be enumerated by determining which cohomology classes in \(\sch{2}{\bb{T}^6,\rqt{\bb{Z}}{2\bb{Z}}}\) can arise as the second Stiefel--Whitney class of flat, orientable, rank 3 vector bundles.  This is accomplished in the following result.
\begin{Prop}\label{w2-char}
Let \(w\in \sch{2}{\bb{T}^6,\rqt{\bb{Z}}{2\bb{Z}}}\). Then, \(w\) is the second Stiefel--Whitney class of a flat, orientable, rank 3 vector bundle over \(\bb{T}^6\) \iff\ it can be written as \(w = a\cup b\) for some classes \(a,b\in \sch{1}{\bb{T}^6,\rqt{\bb{Z}}{2\bb{Z}}}\).
\end{Prop}

\begin{proof}
Firstly, from the K\^^22{u}nneth formula (\cite[Thm.\ 3.15]{AT}), one may prove by induction that for any \(n\):
\ew
\sch{*}{\bb{T}^n;\rqt{\bb{Z}}{2\bb{Z}}} \cong \ww{*}\lt(\rqt{\bb{Z}}{2\bb{Z}}\rt)^n,
\eew
where:
\ew
\ww{*}\lt(\rqt{\bb{Z}}{2\bb{Z}}\rt)^n = \rqt{\Ts^*\lt(\rqt{\bb{Z}}{2\bb{Z}}\rt)^n}{\lt\{\ga \ts \ga = 0 ~\m|~ \ga \in \lt(\rqt{\bb{Z}}{2\bb{Z}}\rt)^n\rt\}}
\eew
as usual. (Here, the tensor product is taken over \(\rqt{\bb{Z}}{2\bb{Z}}\).)

Now, by \cite[Thm.\ 3.2]{VBOT&NS}, every flat vector bundle over a torus is isomorphic to a Whitney sum of flat line bundles.  (The statement of Thm.\ 3.2 in \cite{VBOT&NS} does not include the fact that the line bundles are themselves flat, however this result follows from the proof given on p.\ 273 {\it op.\ cit.}.) Thus, consider \(\bb{E} = \ell_1 \ds \ell_2 \ds \ell_3\) for \(\ell_i\) flat line bundles over \(\bb{T}^6\). The requirement that \(\bb{E}\) be orientable is equivalent to the requirement that \(w_1(\bb{E}) = 0\), i.e.\ that:
\ew
w_1(\ell_3) = w_1(\ell_1) + w_1(\ell_2).
\eew
Using this relation, one may compute that:
\ew
w_2(\bb{E}) &= w_1(\ell_1)\cup w_1(\ell_2) + w_1(\ell_1)\cup \lt[w_1(\ell_1) + w_1(\ell_2)\rt] + w_1(\ell_2) \cup \lt[w_1(\ell_1) + w_1(\ell_2)\rt]\\
&= w_1(\ell_1) \cup w_1(\ell_2)
\eew
since \(w_1(\ell_1)^{\cup2} = w_1(\ell_2)^{\cup2} = 0\), and so \(w_2(\bb{E})\) has the form claimed.

Conversely, suppose given a class \(w\in \sch{2}{\bb{T}^6,\rqt{\bb{Z}}{2\bb{Z}}}\) such that \(w = a\cup b\) for \(a,b\in \sch{1}{\bb{T}^6,\rqt{\bb{Z}}{2\bb{Z}}}\).  Using the diagram:
\ew
\bcd
\Hom\lt(\pi_1\lt(\bb{T}^6\rt),O(1)\rt) \cong \Hom\lt(H_1(\bb{T}^6;\bb{Z}),\rqt{\bb{Z}}{2\bb{Z}}\rt) \ar[dd, ^^22 \mLa{\cong} ^^22] \ar[rd, ^^22 \mLa{\cong} ^^22] &\\
& \sch{1}{\bb{T}^6,\rqt{\bb{Z}}{2\bb{Z}}}\\
\rqt{\text{Flat bundles}}{\text{isomorphism}} \ar[ur, ^^22 \mLa{w_1} ^^22] &
\ecd
\eew
there exist flat line bundles \(\ell_1\) and \(\ell_2\) over \(\bb{T}^6\) such that \(w_1(\ell_1) = a\) and \(w_1(\ell_2) = b\).  Then, \(w_1(\ell_1 \ts \ell_2) = a+b\) and thus:
\ew
w_1[\ell_1 \ds \ell_2 \ds (\ell_1 \ts \ell_2)] = a + b + (a+b) = 0,
\eew
implying that \(\bb{E} = \ell_1 \ds \ell_2 \ds (\ell_1 \ts \ell_2)\) is orientable. Moreover, the earlier calculation also shows that \(w_2(\bb{E}) = a\cup b\), as required.

\end{proof}

Using \pref{w2-char}, one can count the number of distinct homotopy classes of extendible \slr\ 3-forms produced by Construction \ref{flat-bun-cstr}.  Firstly, note that there is a bijective correspondence between the non-zero elements of \(\sch{2}{\bb{T}^6,\rqt{\bb{Z}}{2\bb{Z}}}\) and the elements of \(\bb{P}_{\bb{Z}/2}\lt(\sch{2}{\bb{T}^6,\rqt{\bb{Z}}{2\bb{Z}}}\rt)\).  Since:
\ew
\sch{2}{\bb{T}^6,\rqt{\bb{Z}}{2\bb{Z}}} \cong \ww{2}\sch{1}{\bb{T}^6,\rqt{\bb{Z}}{2\bb{Z}}},
\eew
where the exterior-square is taken over the base field \(\rqt{\bb{Z}}{2\bb{Z}}\), it follows that the set of non-zero second Stiefel--Whitney classes of flat orientable rank 3 bundles over \(\bb{T}^6\) is precisely the image of the `Pl\"{u}cker-type' embedding:
\ew
\bcd[row sep = 0pt]
\Gr_2\lt(\sch{1}{\bb{T}^6,\rqt{\bb{Z}}{2\bb{Z}}}\rt) \ar[r, hook]& \bb{P}_{\bb{Z}/2}\lt(\ww{2}\sch{1}{\bb{T}^6,\rqt{\bb{Z}}{2\bb{Z}}}\rt)\\
\Pi \ar[r, maps to]& \ww{2}\Pi\hs{1.5pt}.
\ecd
\eew
Since \(\Gr_2\lt(\sch{1}{\bb{T}^6,\rqt{\bb{Z}}{2\bb{Z}}}\rt)\) contains 651 elements (see \aref{Counting-Planes}) Construction \ref{flat-bun-cstr} generates \(652 = 651 + 1\) distinct homotopy classes of extendible \slr\ 3-forms over \(\bb{T}^6\) (the extra case arising from \(w_2(\bb{E}) = 0\)).  Moreover, by applying the homotopy equivalences in \eref{HE} (and since extendibility is a homotopy invariant), Construction \ref{flat-bun-cstr} implies the existence of 652 distinct homotopy classes of closed, extendible \slr\ 3-forms on \(\bb{T}^6\), and likewise for extendible \slr\ 3-forms in any given degree 3 cohomology class.\\

\appendix

\section{Enumerating \(k\)-planes in \(\lt(\rqt{\bb{Z}}{2\bb{Z}}\rt)^n\)}\label{Counting-Planes}

The aim of this appendix is to prove the following result.
\begin{Prop}\label{no-planes}
Let \(\bb{F}\) be a finite field.  Recall the \(q\)-Pochhammer symbol:
\ew
(a;q)_n = \prod_{i=0}^{n-1} (1-aq^i),
\eew
where \(a\in\bb{R}\), \(q\in(0,1)\) and \(n\in\bb{N}\).  Then:
\ew
\lt|\Gr_k(\bb{F}^n)\rt| = N^{k(n-k)}\frac{\lt(\frac{1}{N};\frac{1}{N}\rt)_n}{\lt(\frac{1}{N};\frac{1}{N}\rt)_k\lt(\frac{1}{N};\frac{1}{N}\rt)_{(n-k)}}.
\eew
\end{Prop}

Initially, let \(\bb{F}\) be an arbitrary field.
\begin{Lem}\label{Gr&GL}
\ew
\Gr_k(\bb{F}^n) \cong \rqt{\GL(n;\bb{F})}{\big[\GL(k;\bb{F})\x\GL(n-k;\bb{F})\big]\sdp\End\lt(\bb{F}^{n-k},\bb{F}^k\rt)}
\eew
where the multiplication on \(\big[\GL(k;\bb{F})\x\GL(n-k;\bb{F})\big]\sdp\End\lt(\bb{F}^{n-k},\bb{F}^k\rt)\) is given by:
\ew
(A,B;C)\cdot(A',B';C') = (AA',BB',AC'B^{-1} +C).
\eew
Here, \(A,A'\in\GL(k;\bb{F})\), \(B,B'\in\GL(n-k,\bb{F})\) and \(C,C'\in\End\lt(\bb{F}^{n-k},\bb{F}^k\rt)\).
\end{Lem}

\begin{proof}
Clearly \(\GL(n,\bb{F})\) acts transitively on \(\Gr_k(\bb{F}^n)\). Thus, fix \(\Pi\in\Gr_k(\bb{F}^n)\) and choose an algebraic complement \(\Pi'\) to \(\Pi\) in \(\bb{F}^n\). \Wrt\ the splitting \(\Pi \ds \Pi' \cong \bb{F}^n\),  the stabiliser in \(\GL(n,\bb{F})\) of \(\Pi\) consists precisely of those linear maps of the form:
\ew
\bpm
A & D\\
    & B
\epm,
\eew
where \(A\in\GL(k;\bb{F})\), \(B\in\GL(n-k;\bb{F})\) and \(D\in\End\lt(\bb{F}^{n-k},\bb{F}^k\rt)\).  The map \((A,B,D) \mt (A,B,DB^{-1})\) defines an isomorphism from the stabiliser of \(\Pi\) to the group \(\big[\GL(k;\bb{F})\x\GL(n-k;\bb{F})\big]\sdp\End\lt(\bb{F}^{n-k},\bb{F}^k\rt)\), as defined above.

\end{proof}

Now, restrict attention to the case where \(\bb{F}\) is a finite field, say \(|\bb{F}| = N\).
\begin{Lem}\label{noP}
Write \(\bb{FP}^{n-1} = \Gr_1(\bb{F}^n)\). Then:
\ew
\lt|\bb{F}\bb{P}^{n-1}\rt| = \frac{N^n-1}{N-1}.
\eew
\end{Lem}

\begin{proof}
Every non-zero element in \(\bb{F}^n\) (of which there are \(N^n-1\)) determines a unique line through the origin, however each line through the origin contains precisely \(N-1\) non-zero points. The result follows.

\end{proof}

\begin{Lem}\label{noGL}
\e\label{no-GL}
\lt|\GL(n,\bb{F})\rt| = N^{(n^2)}\prod_{i=1}^n \lt(1-\lt(\frac{1}{N}\rt)^i\rt) = N^{n^2}\lt(\frac{1}{N};\frac{1}{N}\rt)_n.
\ee
\end{Lem}
\vs{5mm}

\begin{proof}
Proceed by induction. In the case \(n=1\), \(\GL(1,\bb{F})\) consists of the non-zero elements of \(\bb{F}\) and thus has size \((N-1)\), as required.  In general, using \lref{Gr&GL}, one sees that:
\ew
\frac{\lt|\GL(n+1;\bb{F})\rt|}{\lt|\GL(1;\bb{F})\rt| \x \lt|\GL(n;\bb{F})\rt| \x \lt|\lt(\bb{F}^n\rt)^*\rt|} = \lt|\bb{FP}^n\rt|.
\eew
Thus, by using \lref{noP} together with \(\lt|\lt(\bb{F}^n\rt)^*\rt| = N^n\), one sees inductively that:
\ew
\lt|\GL(n+1;\bb{F})\rt| &= |(\bb{F}^n)^*| \1 |\GL(1;\bb{F})| \1 |\bb{FP}^n| \1 \lt|\GL(n;\bb{F})\rt|\\
&= N^n \1 (N-1) \1 \frac{N^{n+1}-1}{N-1} \1 N^{(n^2)}\lt(\frac{1}{N};\frac{1}{N}\rt)_n\\
&= N^{n^2 + 2n + 1} \lt(1 - \lt(\frac{1}{N}\rt)^{n+1}\rt)\lt(\frac{1}{N};\frac{1}{N}\rt)_n\\
& = N^{(n+1)^2}\lt(\frac{1}{N};\frac{1}{N}\rt)_{n+1},
\eew
as required.

\end{proof}

I now prove \pref{no-planes}.

\begin{proof}
Using \lref{Gr&GL}, one computes that:
\ew
\lt|\Gr_k(\bb{F}^n)\rt| = \frac{\lt|\GL(n;\bb{F})\rt|}{\lt|\GL(k;\bb{F})\rt| \x \lt|\GL(n-k;\bb{F})\rt| \x \lt|\End\lt(\bb{F}^{n-k},\bb{F}^k\rt)\rt|}.
\eew
Substituting the result of \lref{noGL} together with \(\lt|\End\lt(\bb{F}^{n-k},\bb{F}^k\rt)\rt| = N^{k(n-k)}\) yields:
\ew
\lt|\Gr_k(\bb{F}^n)\rt| = \frac{N^{n^2}\lt(\frac{1}{N};\frac{1}{N}\rt)_n}{N^{k^2}\lt(\frac{1}{N};\frac{1}{N}\rt)_k \x N^{(n-k)^2}\lt(\frac{1}{N};\frac{1}{N}\rt)_{(n-k)} \cdot N^{k(n-k)}}.
\eew
The result follows from the identity \(k^2 + (n-k)^2 + k(n-k) = n^2 - k(n-k)\).

\end{proof}

In particular, the number of 2-planes in 6-dimensional space over \(\bb{F} = \rqt{\bb{Z}}{2\bb{Z}}\) is:
\ew
2^{2\cdot4}\frac{\lt(\frac{1}{2};\frac{1}{2}\rt)_6}{\lt(\frac{1}{2};\frac{1}{2}\rt)_2\lt(\frac{1}{2};\frac{1}{2}\rt)_4} = 651,
\eew
as claimed in \sref{slr}.

~\vs{5mm}

\noindent Laurence H.\ Mayther\\
University of Cambridge\\
United Kingdom\\
{\it lhm32@cam.ac.uk}

\end{document}

%% file: Festino_1pt0.tex
\usepackage{adjustbox,array,bigints,cancel,color,comment,extpfeil,mathdots,mathrsfs,mathtools,MnSymbol,multicol,scalerel,setspace,tcolorbox,tikz,tikz-cd,upgreek}
\usepackage[fontsize = 10pt]{fontsize}
\usepackage[left = 2.5cm, right = 2.5cm, top = 2.5cm, bottom = 2.5cm]{geometry}
\usepackage{hyperref}
\usetikzlibrary{patterns}
\numberwithin{equation}{section}

\theoremstyle{plain}

\newtheorem{Cor}[equation]{Corollary}

\newtheorem{Lem}[equation]{Lemma}

\newtheorem{Prop}[equation]{Proposition}

\newtheorem{Thm}[equation]{Theorem}

\theoremstyle{definition}

\newtheorem{Cstr}[equation]{Construction}
\newtheorem{Defn}[equation]{Definition}

\newtheorem{Exs}[equation]{Examples}

\theoremstyle{remark}

\newtheorem{Rks}[equation]{Remarks}

\theoremstyle{plain}
\newtheorem*{Cl*}{Claim}
\newtheorem*{Conj*}{Conjecture}
\newtheorem*{Lem*}{Lemma}
\newtheorem*{Prop*}{Proposition}
\newtheorem*{Q*}{Question}
\newtheorem*{Schol*}{Scholium}
\newtheorem*{SubCl*}{Subclaim}
\newtheorem*{Thm*}{Theorem}

\theoremstyle{definition}
\newtheorem*{Cond*}{Condition}
\newtheorem*{Cstr*}{Construction}
\newtheorem*{Defn*}{Definition}
\newtheorem*{Ex*}{Example}
\newtheorem*{Exs*}{Examples}
\newtheorem*{Md*}{Method}
\newtheorem*{Nt*}{Notation}
\newtheorem*{Pty*}{Property}

\theoremstyle{remark}

\newtheorem*{Rk*}{Remark}
\newtheorem*{Rks*}{Remarks}
\newtheorem*{A-d}{Aside}

\newcommand{\cag}{\begin{equation}\begin{gathered}}
\newcommand{\caag}{\end{gathered}\end{equation}}
\newcommand{\caw}{\begin{equation*}\begin{gathered}}
\newcommand{\caaw}{\end{gathered}\end{equation*}}
\newcommand{\e}{\begin{equation}\begin{aligned}}
\newcommand{\ee}{\end{aligned}\end{equation}}
\newcommand{\ew}{\begin{equation*}\begin{aligned}}
\newcommand{\eew}{\end{aligned}\end{equation*}}

\newcommand{\bcd}{\begin{tikzcd}}
\newcommand{\ecd}{\end{tikzcd}}
\newcommand{\bma}{\begin{matrix}}
\newcommand{\ema}{\end{matrix}}
\newcommand{\bpm}{\begin{pmatrix}}
\newcommand{\epm}{\end{pmatrix}}
\newcommand{\bvm}{\begin{vmatrix}}
\newcommand{\evm}{\end{vmatrix}}

\newcommand{\nts}{\begin{tcolorbox}}
\newcommand{\ntss}{\end{tcolorbox}}

\newcommand{\aref}[1]{Appendix \ref{#1}}

\newcommand{\cref}[1]{Corollary \ref{#1}}

\newcommand{\dref}[1]{Definition \ref{#1}}

\newcommand{\eref}[1]{eqn.\hspace{0.6mm}(\ref{#1})}

\newcommand{\exref}[1]{Example \ref{#1}}

\newcommand{\lref}[1]{Lemma \ref{#1}}

\newcommand{\pref}[1]{Proposition \ref{#1}}
\newcommand{\prefs}[1]{Propositions \ref{#1}}

\newcommand{\sref}[1]{\S\ref{#1}}

\newcommand{\tref}[1]{Theorem \ref{#1}}

\DeclareMathSizes{10}{10}{8}{7}

\newcommand{\mss}[1]{\mbox{\scriptsize \(#1\)}}
\newcommand{\mfn}[1]{\mbox{\footnotesize \(#1\)}}
\newcommand{\msm}[1]{\mbox{\small \(#1\)}}

\newcommand{\mLa}[1]{\mbox{\Large \(#1\)}}

\newcommand{\bb}[1]{\mathbb{#1}}
\newcommand{\cal}[1]{\mathscr{#1}}
\newcommand{\fr}[1]{\mathfrak{#1}}
\newcommand{\mb}[1]{\mbox{\boldmath \(#1\)}}
\newcommand{\mc}[1]{\mathcal{#1}}

\newcommand{\et}{\hspace{5mm}\text{and}\hspace{5mm}}
\newcommand{\hs}[1]{\hspace{#1}}

\newcommand{\vs}[1]{\vspace{#1}}

\DeclareMathSymbol{\Alpha}{\mathalpha}{operators}{"41}
\DeclareMathSymbol{\Beta}{\mathalpha}{operators}{"42}
\DeclareMathSymbol{\Epsilon}{\mathalpha}{operators}{"45}
\DeclareMathSymbol{\Zeta}{\mathalpha}{operators}{"5A}
\DeclareMathSymbol{\Eta}{\mathalpha}{operators}{"48}
\DeclareMathSymbol{\Iota}{\mathalpha}{operators}{"49}
\DeclareMathSymbol{\Kappa}{\mathalpha}{operators}{"4B}
\DeclareMathSymbol{\Mu}{\mathalpha}{operators}{"4D}
\DeclareMathSymbol{\Nu}{\mathalpha}{operators}{"4E}
\DeclareMathSymbol{\Omicron}{\mathalpha}{operators}{"4F}
\DeclareMathSymbol{\Rho}{\mathalpha}{operators}{"50}
\DeclareMathSymbol{\Tau}{\mathalpha}{operators}{"54}
\DeclareMathSymbol{\Chi}{\mathalpha}{operators}{"58}
\DeclareMathSymbol{\omicron}{\mathord}{letters}{"6F}

\newcommand{\al}{\alpha}

\newcommand{\ga}{\gamma}
\newcommand{\de}{\delta}

\renewcommand{\th}{\theta}

\newcommand{\io}{\iota}

\newcommand{\vpi}{\varpi}
\newcommand{\rh}{\rho}
\newcommand{\vrh}{\varrho}
\newcommand{\si}{\sigma}
\newcommand{\vsi}{\varsigma}

\newcommand{\ph}{\phi}
\newcommand{\vph}{\upvarphi}

\newcommand{\ch}{\chi}

\newcommand{\om}{\omega}

\newcommand{\Th}{\Theta}

\newcommand{\Om}{\Omega}

\newcommand{\<}{\langle}
\newcommand{\?}{\rangle}

\newcommand{\Ann}{\operatorname{Ann}}

\newcommand{\ds}{\oplus}
\newcommand{\End}{\operatorname{End}}

\newcommand{\Hom}{\operatorname{Hom}}
\newcommand{\Id}{\operatorname{Id}}

\newcommand{\Iso}{\operatorname{Iso}}

\newcommand{\rqt}[2]{\left.\raisebox{1mm}{\(#1\)}\middle/\raisebox{-1mm}{\(#2\)}\right.}
\newcommand{\ts}{\otimes}

\newcommand{\Ts}{\bigotimes}
\newcommand{\x}{\times}

\newcommand{\del}{\partial}

\newcommand{\Hocl}[1]{\overset{\circ}{H^{#1}}_{\kern-1.9mm\cl}}

\newcommand{\lop}{\left\|\kern-1.30mm\left\|}
\newcommand{\op}{\|\kern-1.30mm\|}
\newcommand{\rop}{\right\|\kern-1.30mm\right\|}
\newcommand{\SI}{\operatorname{\cal{I}\kern-1.5pt nd}}

\newcommand{\cc}{\subseteq}

\newcommand{\IL}{\varprojlim}

\newcommand{\mt}{\mapsto}

\newcommand{\osr}{\backslash}
\newcommand{\oto}[1]{\xrightarrow{#1}}
\newcommand{\pc}{\subset}

\newcommand{\CL}{\mathcal{C}l}
\newcommand{\cl}{\mathrm{closed}}

\newcommand{\dd}{\mathrm{d}}

\newcommand{\dR}[1]{H^{#1}_{\operatorname{dR}}}

\newcommand{\emb}{\hookrightarrow}

\newcommand{\Gr}{\mathrm{Gr}}

\newcommand{\hk}{\righthalfcup}

\newcommand{\Hs}{\raisebox{1pt}{\mss{\bigstar}}}

\newcommand{\K}{\mathrm{K}}

\newcommand{\M}{\mathrm{M}}

\newcommand{\N}{\mathrm{N}}
\newcommand{\oGr}{\widetilde{\mathrm{\Gr}}}

\newcommand{\Op}{\mathcal{O}p}

\newcommand{\sch}[2]{\operatorname{H}^{#1}\left(#2\right)}

\newcommand{\sph}{\widetilde{\phi}}

\renewcommand{\ss}[2][{}]{\bigodot{\hspace{-1mm}}^{#2}_{#1}\hspace{0.6mm}}

\newcommand{\svph}{\widetilde{\upvarphi}}
\newcommand{\svps}{\widetilde{\uppsi}}
\newcommand{\T}{\mathrm{T}}

\newcommand{\w}{\wedge}
\newcommand{\ww}[2][{}]{\bigwedge{\hspace{-1mm}}^{#2}_{\hspace{1mm}#1}\hspace{0.1mm}}

\newcommand{\1}{\cdot}
\newcommand{\bin}{\binom}

\renewcommand{\ge}{\geqslant}
\newcommand{\gl}{\hspace{0.4mm}\raisebox{0.8mm}{\(>\)}\kern-1.8mm\raisebox{-0.8mm}{\(<\)}\hspace{0.4mm}}
\newcommand{\gle}{\hspace{0.4mm}\raisebox{1.2mm}{\(\ge\)}\kern-1.8mm\raisebox{-1.2mm}{\(\le\)}\hspace{0.4mm}}
\renewcommand{\le}{\leqslant}
\newcommand{\pt}{\bullet}

\newcommand{\g}{\(\mathrm{G}_2\)}
\newcommand{\Gg}{\mathrm{G}}
\newcommand{\GL}{\operatorname{GL}}

\newcommand{\sdp}{\ltimes}
\newcommand{\sg}{\(\widetilde{\mathrm{G}}_2\)}
\newcommand{\SL}{\operatorname{SL}}
\newcommand{\slc}{\(\operatorname{SL}(3;\mathbb{C})\)}
\newcommand{\slr}{\(\operatorname{SL}(3;\mathbb{R})^2\)}
\newcommand{\SO}{\operatorname{SO}}

\newcommand{\Spin}{\operatorname{Spin}}
\newcommand{\Stab}{\operatorname{Stab}}
\newcommand{\SU}{\operatorname{SU}}
\newcommand{\Un}{\operatorname{U}}

\newcommand{\acs}{almost complex structure}
\renewcommand{\iff}{if and only if}

\newcommand{\wlg}{without loss of generality}
\newcommand{\Wrt}{With respect to}
\newcommand{\wrt}{with respect to}

\newcommand{\ol}[1]{\overline{#1}}

\newcommand{\lt}{\left}
\newcommand{\m}{\middle}

\newcommand{\rt}{\right}

\newcommand{\tld}{\widetilde}